\PassOptionsToPackage{colorlinks=true,linkcolor=blue,citecolor=black,urlcolor=black}{hyperref}
\documentclass[final,3p,times,sort&compress,lefttitle]{elsarticle}

\usepackage[T1]{fontenc}
\usepackage[utf8]{inputenc}
\usepackage{amsmath}
\usepackage{amsthm}
\usepackage{amssymb}
\usepackage{tikz}
\usepackage{graphicx}
\usepackage[section]{placeins}
\usepackage{doi}
\usepackage{mleftright}
\definecolor{zzttqq}{rgb}{0.6,0.2,0}
\usepackage{booktabs}

\newcommand{\set}[1]{\ensuremath{\mleft\{ #1 \mright\}}}
\newcommand{\apply}[1]{\ensuremath{\mleft( #1 \mright)}}

\newcommand{\crd}[1]{\ensuremath{\mleft\lvert #1 \mright\rvert}}
\newcommand{\tup}[2]{\mleft(#1_1,\dotsc,#1_{#2}\mright)}
\newcommand{\Lp}[1][n]{\ensuremath{\mathcal{L}_{#1}}}
\newcommand{\Ji}[1][n]{\ensuremath{\mathcal{J}\apply{\Lp[#1]}}}
\newcommand{\Mi}[1][n]{\ensuremath{\mathcal{M}\apply{\Lp[#1]}}}
\newcommand{\Jirr}[1][\mathbb{L}]{\ensuremath{\mathcal{J}\apply{#1}}}
\newcommand{\Mirr}[1][\mathbb{L}]{\ensuremath{\mathcal{M}\apply{#1}}}
\newcommand{\Bh}[2][n]{\ensuremath{\mathcal{H}_{#1}(#2)}}
\newcommand{\Bw}[2][n]{\ensuremath{\mathcal{W}_{#1}(#2)}}

\newcommand{\tupone}[2]{(#1,#1,\dotsc,\overset{#2}{#1})}
\newcommand{\tupdos}[4]{(#1,#1,\dotsc,\overset{#2}{#1},#3,#3,\dotsc,\overset{#4}{#3})}

\newcommand{\gupm}[2][\Mi]{\ensuremath{#2{\uparrow_{#1}}}}
\newcommand{\mdownj}[2][\Ji]{\ensuremath{#2{\downarrow_{#1}}}}
\newcommand{\upset}[1]{\mathop{\uparrow}#1}
\let\gup=\upset
\newcommand{\downset}[1]{\mathop{\downarrow}#1}
\let\mdown=\downset
\DeclareMathOperator{\Part}{Part}
\renewcommand{\land}{\mathbin{\&}}%
\newcommand{\disjointunion}{\mathbin{\dot{\cup}}}%
\newcommand{\diff}{\ensuremath{\, \mathrel{\mathop:}\Longleftrightarrow \, }}
\DeclareMathOperator{\hgt}{ht}
\DeclareMathOperator{\len}{len}

\providecommand{\Runterpfeil}{\mathrel{\swarrow}}
\providecommand{\DownArrow}{\Runterpfeil}
\DeclareRobustCommand{\nRunterpfeil}{\mathrel{\ooalign{$\Runterpfeil$\cr\hidewidth\raise.3ex\hbox{${\scriptstyle-\mkern5mu}$}}}}
\providecommand{\Hochpfeil}{\mathrel{\nearrow}}
\providecommand{\UpArrow}{\Hochpfeil}
\DeclareRobustCommand{\nHochpfeil}{\mathrel{\ooalign{$\Hochpfeil$\cr\hidewidth${\scriptstyle-\mkern8.5mu}$}}}

\DeclareRobustCommand{\Doppelpfeil}{\mathrel{\ooalign{$\nearrow$\cr\hidewidth$\swarrow$}}}
\providecommand{\DoubleArrow}{\Doppelpfeil}
\def\crossingoutchar%
{\hbox{\bgroup\fontencoding{OT1}\selectfont\char'40\egroup}}%
\providecommand{\notI}{\mathrel{\mbox{\rlap{\crossingoutchar}%
$I$\hspace*{-0.25em}\raisebox{.27ex}{\crossingoutchar}}}}%
\newcommand{\defeq}{\mathrel{\mathop:}=}
\newcommand{\eqdef}{\mathrel{\mathopen={\mathclose:}}}
\DeclareRobustCommand{\lset}[2]{\ensuremath{\mleft\{\mleft.#1\ \vphantom{#2}\mright|\ #2\mright\}}}

\newcommand{\N}{\mathbb{N}}
\newcommand{\Nplus}{\mathbb{N}_+}
\newcommand{\La}{\mathbb{L}}
\newcommand{\Pa}{\mathbb{P}}
\newcommand{\Qa}{\mathbb{Q}}
\newcommand{\K}{\mathbb{K}}

\newtheorem{lemma}{Lemma}[section]
\newtheorem{theorem}[lemma]{Theorem}
\newtheorem{corollary}[lemma]{Corollary}
\newtheorem{proposition}[lemma]{Proposition}

\theoremstyle{definition}
\newtheorem{definition}[lemma]{Definition}
\newtheorem{remark}[lemma]{Remark}
\newtheorem{example}[lemma]{Example}

\renewcommand{\figurename}{Fig.}

\begin{document}
\begin{frontmatter}
\title{Arrow Relations in Lattices of Integer Partitions}

\author[TTU]{Asma'a Almazaydeh}
\ead{aalmazaydeh@ttu.edu.jo}
\affiliation[TTU]{organization={Department of Mathematics, Tafila Technical University},
            addressline={PO Box~179},
            postcode={66110},
            city={Tafila},
            country={Jordan}}
\author[TUWien]{Mike Behrisch}
\ead{behrisch@logic.at}
\affiliation[TUWien]{organization={Institute of Discrete Mathematics and Geometry, TU Wien},
            addressline={Wiedner Hauptstr.~8--10},
            postcode={1040},
            city={Vienna},
            country={Austria}}
\author[ITAM]{Edith Vargas-Garc\'\i{}a\fnref{AMC}}
\ead{edith.vargas@itam.mx}
\affiliation[ITAM]{organization={Department of Mathematics, ITAM},
            addressline={R\'{\i}o Hondo~1},
            city={Ciudad de M\'exico},
            postcode={CP~01080},
            country={Mexico}}
\author[ITAM]{Andreas Wachtel\corref{corauth}\fnref{AMC}}
\ead{andreas.wachtel@itam.mx}
\cortext[corauth]{Corresponding author}
\fntext[AMC]{This author gratefully acknowledges support by the A\-so\-cia\-ci\'on Me\-xi\-ca\-na de Cul\-tu\-ra A.C.}

\begin{abstract}
We give a complete characterisation of the single and double
arrow relations of the standard context~$\K(\Lp)$ of the lattice~$\Lp$
of partitions of any positive integer~$n$ under the dominance order,
thereby addressing an open question of Ganter, 2022.
\end{abstract}

\begin{graphicalabstract}
~
\vspace{3cm}

\noindent
\includegraphics[width=\linewidth]{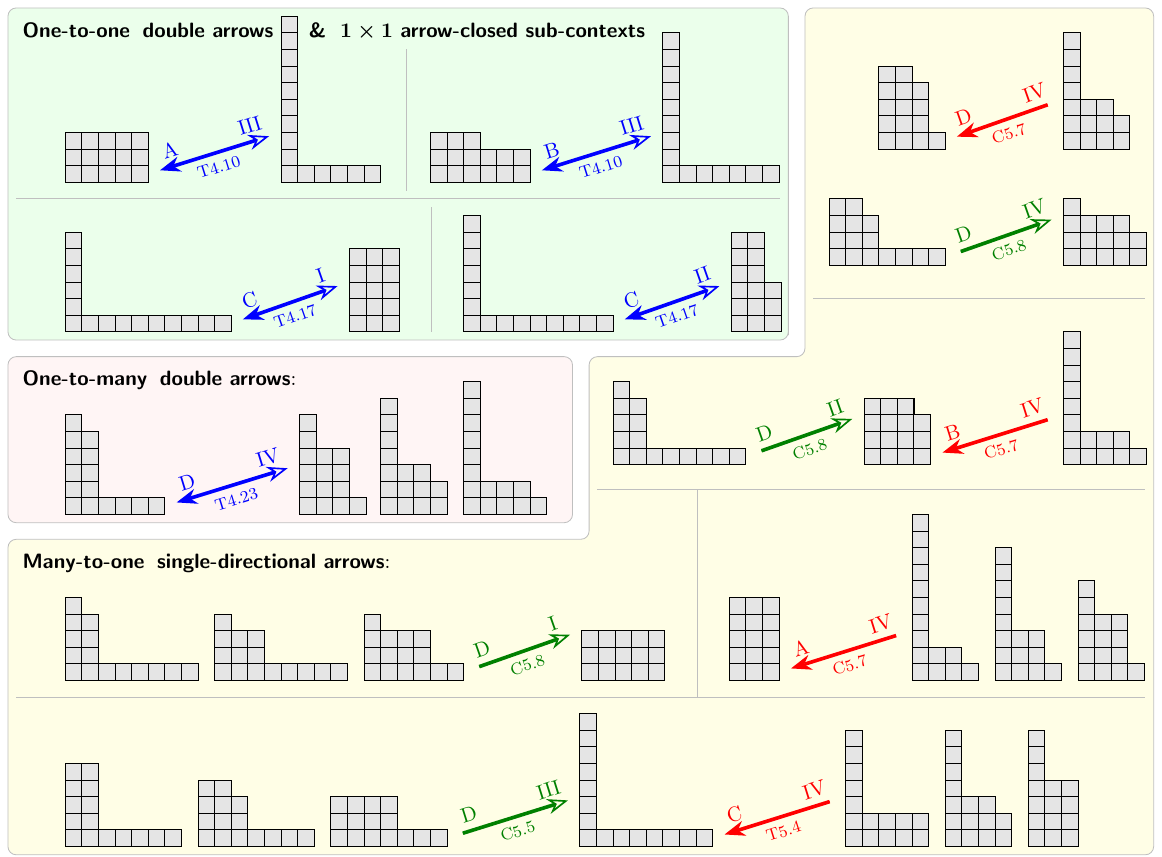}%
\end{graphicalabstract}

\begin{highlights}%
\item Type~\ref{typeD} partitions allow up-arrows, which fail to be
      down-arrows, to all types~\ref{typeI}--\ref{typeIV}.
\item Type~\ref{typeIV} partitions allow down-arrows, which fail to be
      up-arrows, to all types~\ref{typeA}--\ref{typeD}.
\item Type~\ref{typeD} partitions have double arrows only to those of type~\ref{typeIV} and vice
      versa.
\item Partitions of types~\ref{typeA}--\ref{typeC} are double arrow
      related only to types~\ref{typeI}--\ref{typeIII} and vice versa.
\item For $n\geq3$ there are exactly $2n-4$ one-generated arrow-closed
      $(1\times1)$-subcontexts of~$\K(\Lp)$.
\end{highlights}

\end{frontmatter}

\section{Introduction}\label{sect:intro}

      Integer partitions have captivated mathematicians
      for centuries, starting as early as 1674 with Leibniz
      investigating the number~$p(n)$ of ways in which a natural
      number~$n$ can be partitioned, that is, expressed as a sum of a
      non-increasing sequence of positive integer summands,
      see~\cite[p.~37]{mahnke1912leibniz}.
      Recursive presentations of~$p(n)$, for example, following
      from Euler's pentagonal number theorem, are well known, and the
      search for more explicit formul\ae\ or approximations for~$p(n)$
      culminated with the celebrated asymptotic expressions given by
      Hardy and Ramanujan~\cite{HardyRamanujaniAsymptoticFormulaeCombinatoryAnalysis1918}
      and with Rademacher's representations by convergent
      series~\cite{Rademacher1937PartitionFunction,Rademacher1943SeriesExpansionOfThePartitionFunction}.
      \par
      The partitions of a given integer~$n\in\N$ can be ordered by
      dominance, that is, by pointwise comparing their sequences of
      partial sums; the resulting ordered set carries the structure of
      a finite (hence complete) lattice~$\Lp$, see~\cite{Brylawski}.
      The cardinalities of these lattices~$\Lp$, that is, the
      numbers~$p(n)=\crd{\Lp}$ of (unrestricted) partitions of~$n$,
      grow fast as~$n$ is increasing, the asymptotics shown
      in~\cite{HardyRamanujaniAsymptoticFormulaeCombinatoryAnalysis1918}
      to be
      $p(n)\sim\frac{\exp\bigl(\pi\sqrt{2n/3}\bigr)}{4n\sqrt{3}}$ for
      $n\to\infty$.
      Their size alone suggests an increasing complexity of the
      lattices~$\Lp$ for larger values of~$n$, wherefore splitting them
      up into smaller parts would be an important step towards
      enhancing our understanding of their structure. Fortunately,
      formal concept analysis and lattice theory offer techniques for
      this task in the form of subdirect representations of complete
      lattices, see, e.g., \cite[Chapter~4]{ganter}.
      These are embeddings of the given lattice~$\Lp$ into a direct product of
      smaller lattices, the subdirect factors, such that for each
      coordinate the corresponding projection is surjective. Of course, the subdirect
      factors may themselves be again subdirectly representable by even
      smaller factors, leading eventually to the concept of subdirect
      irreducibility: a lattice is subdirectly irreducible if in each
      subdirect representation at least one of the coordinate
      projections is not only surjective, but bijective, that is, an
      isomorphism onto the corresponding subdirect factor. The most
      efficient decompositions are hence given by representations using
      subdirectly irreducible factors. Such decompositions exist
      for any doubly founded (in particular any finite) complete
      lattice, such as~$\Lp$, see~\cite[Theorem~18]{ganter}.
      \par
      Formal concept analysis~\cite{ganter} offers a powerful framework to study
      complete lattices~$\mathbb{L}$ (up to isomorphism) as lattices of
      Galois closed sets of a suitable Galois connection between
      completely join-dense and completely meet-dense subsets of the
      lattice. This is part of the basic (or fundamental) theorem of
      formal concept analysis, see~\cite[Theorem~3]{ganter}. The Galois connection is, up to isomorphism,
      induced by the order relation of the lattice between the elements
      of the two dense subsets, and this inducing binary relation is
      usually represented in a tabular form, called formal context. A
      canonically derived complete lattice, the concept lattice, is
      isomorphic (anti-isomorphic) to the lattice of Galois closed sets
      and serves to represent the given complete lattice~$\mathbb{L}$. For
      finite lattices~$\mathbb{L}$ there is, up to isomorphism, a unique
      way of representation, namely through the so-called standard context,
      which is given by the order relation between all completely
      join-irreducible and complete meet-irreducible elements
      of~$\mathbb{L}$, see~\cite[Proposition~12]{ganter}. Formal concept
      analysis further defines binary relations $\Runterpfeil, \Hochpfeil$ and
      ${\Doppelpfeil}={\Runterpfeil}\cap{\Hochpfeil}$ as certain subsets
      of the complement of the relation represented by a formal
      context, cf.~\cite[Definition~25]{ganter}.
      These arrow relations appear in the `empty cells' of the context
      table, and from them one may determine so-called one-generated
      arrow-closed subcontexts. This is done by adding attributes
      (resp.\ objects) pointed to by arrows ensuing from objects
      (resp.\ attributes) already appearing
      in the subcontext until the configuration stabilises.
      According to~\cite[Proposition~62]{ganter}
      the one-generated arrow-closed subcontexts of the standard context
      of a finite lattice~$\mathbb{L}$ give subdirectly irreducible
      concept lattices, and taking sufficiently many of them, one may
      construct a subdirect decomposition of~$\mathbb{L}$,
      see~\cite[Proposition~61]{ganter}.
      A thorough understanding of the arrow relations in the standard
      context is hence a crucial step towards systematically
      obtaining subdirect decompositions of~$\Lp$.
      \par
      For the lattice~$\Lp$ of partitions of an integer~$n\in\N$, the
      sets of completely join-irreducible and completely
      meet-irreducible elements were described by
      Brylawski~\cite{Brylawski}; moreover, a very intuitive
      understanding of the covering relation in~$\Lp$ (and thus of the
      irreducibles) was later given in~\cite{LATAPY}.
      Recursive and non-recursive constructions of the standard
      contexts $\K\apply{\Lp}$ for increasing values of~$n$ were
      studied in~\cite{behrisch2021mas} and~\cite{ganter2022notes}.
      In~\cite{behrisch2021mas}, supported
      by~\cite{BehVarDatasetNonembeddabilityStandardContextsPartLat},
      also the non-embeddability of $\K\apply{\Lp[9]}$ into
      $\K\apply{\Lp[10]}$ was argued, the proof of which was
      later refined in~\cite{FrancoSolorioThesis} for the symmetric
      case, extending previous work
      in~\cite{ganter2022notes,BehDatasetSymmetricEmbeddingsStandardContextsPartLat}.
      \par
      We have computationally determined the standard contexts of~$\Lp$
      for parameters $n\leq 60$ and our results show some curious
      patterns regarding the appearing one-generated arrow-closed
      subcontexts and the corresponding subdirectly irreducible
      factors, cf.~\cite{ParamoPitolThesis} for a limited prospect.
      In order to be able, at a later stage, to substantiate these
      experimental results with rigid proofs, we aim in this paper at a
      complete characterisation of the arrow relations in
      $\K\apply{\Lp}$ for every $n\in\Nplus$, a question that was
      raised in~\cite[p.~40]{ganter2022notes}.
      Our work has partially evolved in parallel
      with~\cite{DiazMirandaThesis}, which also mentions some of the
      more basic results of this article and has, for example, inspired
      the graphical presentation of~$\Lp[11]$ including all arrows
      in \figurename~\ref{fig:lattice11}.
      We first describe all double arrows of $\K(\Lp)$ in
      Theorems~\ref{thm:double-typeAB},
      \ref{thm:double-typeC} and~\ref{thm:double-typeCD}.
      After that we provide characterisations of all down-arrows that fail to
      be up-arrows in Theorems~\ref{thm:TypeABDownIV},
      \ref{thm:TypeCdown1} and~\ref{thm:TypeCDdown}, with a summary in
      Corollary~\ref{cor:TypeABCDdownIVnew}; then we use partition
      conjugation to obtain the dual results, i.e., up-arrows without
      down-arrows, in Corollaries~\ref{corDualTypeCdown1}
      and~\ref{corDualTypeABCDdownIV}.
      Our knowledge regarding arrows is schematically
      summarised in Tables~\ref{tab:DoubleArrowCaracterizationSummary}
      and~\ref{tab:OneDArrowCaracterizationSummary}, and illustrated
      within the lattice~$\Lp[11]$ in \figurename~\ref{fig:lattice11}.
      As a proof of concept we finally show
      in Section~\ref{sect:1-by-1-gen-ACSC}
      how our characterisations can be used to determine and count, for
      each $n\in\Nplus$, all one-generated arrow-closed one-by-one
      subcontexts of $\K(\Lp)$, which correspond to two-element
      subdirectly irreducible lattice factors of~$\Lp$.

\section{Preliminaries}\label{sect:preliminaries}
\subsection{Lattices and ordered sets}
Throughout the text we write $\N\defeq\set{0,1,2,\ldots}$ for the set of
\emph{natural numbers}, and we set $\Nplus\defeq\N\setminus\set{0}$.
An \emph{ordered set} is a pair $\Pa=\apply{P,\leq}$ where~$P$ is a set
and ${\leq}\subseteq P\times P$ is a reflexive, antisymmetric and
transitive binary relation on it.
For $a,b\in P$ we have $a\geq b$ if and only if $b\leq a$, and we write
$a<b$ for $a\leq b$ and $a\neq b$ as usual.
Moreover, we say that \emph{$a$ is covered by~$b$ in~$\Pa$}
(or \emph{$b$ covers~$a$ in~$\Pa$}),
symbolically $a\prec b$, if $a<b$ and $a\leq x <b$ implies $x=a$ for
each $x\in P$.
The \emph{dual} of~$\Pa$ is the ordered set $\apply{P,\geq}$.
An \emph{order-isomorphism} $\varphi\colon\Pa\to\Qa$ is given by a map
$\varphi\colon P\to Q$ such that for all $a,b\in P$ the inequality
$a\leq b$ in~$\Pa$ is equivalent to $\varphi(a)\leq \varphi(b)$
in~$\Qa$. The ordered sets~$\Pa$ and~$\Qa$ are then said to be
\emph{(order-)isomorphic}. An \emph{order-antiisomorphism} or
\emph{dual isomorphism} $\varphi\colon\Pa\to\Qa$ is an order
isomorphism between~$\Pa$ and the dual of~$\Qa$, and~$\Pa$ and~$\Qa$
are then called \emph{antiisomorphic} or \emph{dually isomorphic};
$\Pa$ is \emph{self-dual} (or \emph{autodual}) if it is dually
isomorphic to itself.
Furthermore, we define
$\downset{x}\defeq\set{y\in \Pa \mid y\leq x}$ as the
\emph{principal down-set of $x\in P$}, and
$\upset{x}\defeq\set{y\in \Pa \mid x \leq y}$ as the \emph{principal
up-set of~$x$}.
\par
A \emph{complete lattice} is an ordered set $\La=(L,\leq)$ where any
subset $S\subseteq L$ has a greatest common lower bound $\bigwedge S\in L$
(called \emph{infimum} of~$S$) and a least common upper bound
$\bigvee S\in L$ (called \emph{supremum} of~$S$). It is customary to
write $a_1\wedge\dotsm\wedge a_k$ for $\bigwedge\set{a_1,\dotsc,a_k}$
and $a_1\vee\dotsm\vee a_k$ for $\bigvee\set{a_1,\dotsc,a_k}$ for
finitely many elements $a_1,\dotsc,a_k\in L$.
A subset $D\subseteq L$ is called \emph{(completely) join-dense in~$\La$}
if for every $x\in L$ there is a subset $S\subseteq D$ with
$x = \bigvee S$; $D$ is \emph{(completely) meet-dense in~$\La$} if for
each $x\in L$ there is some $S\subseteq D$ with $x = \bigwedge S$.
We say that $a\in L$ is \emph{completely join-irreducible}, denoted by $\bigvee$-irreducible,
if for all $S\subseteq L$ such that $a=\bigvee S$ we necessarily have
that $a\in S$. Dually, $a\in L$ is \emph{completely meet-irreducible},
or $\bigwedge$-irreducible, if for all $S\subseteq L$ such that
$a=\bigwedge S$ we must have $a\in S$. For a finite non-empty lattice, it is
sufficient to check this condition for the empty and all two-element
subsets~$S$ of~$L$. That is, for finite~$L\neq \emptyset$, an
element~$a\in L$ is completely join-irreducible, if~$a$ is not the
minimum element of~$L$ and for all $b,c\in L$ the condition $a=b\vee c$
implies that $a=b$ or $a=c$. This happens exactly if~$a$ covers exactly
one element below it, cf.~\cite[Proposition~2]{ganter}.
Dually, for a finite lattice, $a\in L$ is completely meet-irreducible
if it is not the top element of~$L$ and for all $b,c\in L$ with
$a=b\wedge c$ it follows that $b=a$ or $c=a$. This condition is met
precisely if~$a$ has exactly one upper cover.
For a complete lattice~$\mathbb{L}$,  we denote by $\Jirr$ and by
$\Mirr$ the sets of all \emph{completely join-irreducible
elements}, and of all \emph{completely meet-irreducible elements}
of~$\mathbb{L}$, respectively. Elements of $\Jirr\cap\Mirr$ are called
\emph{doubly completely irreducible}.
In the lattice diagrams shown in \figurename~\ref{fig:lattice7}
and~\ref{fig:lattice11}, the completely irreducible elements have been
highlighted using half or completely filled nodes.
\par
A (complete) lattice $\mathbb{L}=(L,{\leq})$ is
\emph{supremum-founded} if, for any two $x<y$ from~$L$, the set
$\lset{p\in L}{p\leq y \land p\nleq x}$ contains a $\leq$-minimal
element; the dual property that for any $x<y$ in~$L$ the set
$\lset{p\in L}{x\leq p \land y\nleq p}$ includes a $\leq$-maximal element
is called \emph{infimum-founded}. The lattice~$\mathbb{L}$ is
\emph{doubly founded} if it is both supremum-founded and
infimum-founded, see \cite[p.~33]{ganter}. Every chain-finite and hence
every finite lattice is doubly founded, cf.\
\cite[p.~33 et seqq., \figurename~1.11, p.~35]{ganter}.

\subsection{Notions of formal concept analysis}\label{subsect:fca}
Formal concept analysis (FCA) is a theoretical framework that harnesses
the powers of general abstract Galois theory and the structure theory
of complete lattices for data analysis and many other applications.
At its core lies the notion of a Galois connection between (the power
sets of) sets~$G$ and~$M$, induced by a binary relation
$I\subseteq G\times M$. This data is collected in a \emph{formal
context} $\K=(G,M,I)$, and the elements of~$G$ and~$M$ are given
the interpretative names \emph{objects} and \emph{attributes},
respectively. The set~$I$ is called the \emph{incidence relation}, and
$(g,m)\in I$ is usually written as $g\mathrel{I} m$ and read as
`object~$g$ \emph{has attribute}~$m$'. When~$G$ and~$M$ are finite, the
context~$\K$ is often given as a \emph{cross table}, where the objects
form the rows, the attributes label the columns, and the crosses
represent the characteristic function of~$I$ on $G\times M$. Formal
concept analysis extends the `prime notation' for the Galois
derivatives, which is common in classical Galois
theory~\cite[Chapter~V, Theorem~2.3 et seqq.]{Hungerford2011}, to general
formal contexts $K=(G,M,I)$. For every set~$A\subseteq G$ of objects,
$A'\defeq \set{m\in M\mid\forall g\in A\colon g\mathrel{I} m}$ assigns
to it the set of attributes commonly shared by all objects of~$A$.
Dually, for $B\subseteq M$, the set
$B'\defeq\set{g\in G\mid\forall m\in B\colon g\mathrel{I}m}$
contains exactly those objects possessing all the attributes in~$B$.
A \emph{formal concept} is a pair $(A,B)$ where the \emph{extent}
$A\subseteq G$ and the \emph{intent} $B\subseteq M$ are sets that are
mutually Galois closed: $A=B'$ and $B=A'$. Intents of the form
$B=\set{g}'$ with $g\in G$ are called \emph{object intents} and are
written as $g'$, for short; dually extents $A=\set{m}'\eqdef m'$ with
$m\in M$ are referred to as \emph{attribute extents}. The equivalent
conditions $A_1\subseteq A_2$ and $B_1\supseteq B_2$ define an order
$(A_1,B_1)\leq (A_2,B_2)$ on the set $\mathfrak{B}(\K)$ of all formal
concepts. $\underline{\mathfrak{B}}(\K) =(\mathfrak{B}(\K),\leq)$
becomes a complete lattice under this order, the \emph{concept lattice}
of~$\K$. The fundamental theorem of formal concept
analysis~\cite[Theorem~3]{ganter} states that, in fact, every complete
lattice is a concept lattice, up to isomorphism. Namely,
if~$\mathbb{L}$ is a complete lattice, then
$\mathbb{L}\cong\underline{\mathfrak{B}}(L,L,{\leq})$. For
finite lattices, which are always complete, this construction can be
improved: $\mathbb{L}\cong\underline{\mathfrak{B}}(\K)$,
where $\K=(\Jirr,\Mirr,{\leq})$ is
the \emph{standard context} of
$\mathbb{L}$~\cite[Proposition~12]{ganter}.
This applies to partition lattices~$\Lp$ in particular.
\par
A central notion for this paper will be the \emph{arrow relations} of a
formal context, which fill up some of the empty cells in the cross
table.
\begin{definition}[{\cite[Definition~25]{ganter}}]\label{def:arrow-rels}
If $(G, M, I)$ is a context, $g \in G$ an object, and $m \in M$ an
attribute, we write
\begin{align*}
g \DownArrow m &\diff g \notI m, \text{ and for all } h\in G
\text{ with } g' \subsetneqq h' \text{ we have } h \mathrel{I} m;\\
g \UpArrow m &\diff g \notI m, \text{ and for all } n\in H
\text{ with }m' \subsetneqq n' \text{ we have } g \mathrel{I} n;\\
g\DoubleArrow  m&\diff g \Hochpfeil m \text{ and } g \Runterpfeil m.
\end{align*}

Thus, $g \DownArrow m$ if and only if~$g'$ is maximal among all object intents which
do not contain~$m$; dually we have $g \UpArrow m$ if and only if $m'$
is maximal among all attribute extents which do not contain~$g$.
\end{definition}

We will now derive a useful characterisation of the arrow relations in
standard contexts of doubly founded lattices.
\begin{remark}\label{rem:char-arrows-std-cxt}
Consider the (standard) context
$\K(\mathbb{L})=\apply{\Jirr,\Mirr,\mathord{\leq}}$ of any complete
lattice $\mathbb{L}=(L,\mathord{\leq})$.
Note that for $g\in \Jirr$ and $m\in\Mirr$ we have
\begin{alignat*}{3}
g' &= \lset{m\in\Mirr}{g\leq m}
   &&= \upset{g}\cap \Mirr &&\eqdef \gupm[\Mirr]{g},\\
m' &= \lset{g\in\Jirr}{g\leq m}
   &&= \downset{m}\cap \Jirr &&\eqdef \mdownj[\Jirr]{m}.
\end{alignat*}
This allows us to reformulate the definition of the arrow relations
of~$\K(\mathbb{L})$ in terms of the up-sets and down-sets
of~$\mathbb{L}$. Consider again $g\in \Jirr$ and $m\in \Mirr$.
Then we have
\begin{alignat*}{4}
g \Runterpfeil m
&\stackrel{I\defeq{\leq}}{\iff}  g \nleq m \land
  \forall h\in \Jirr\setminus\set{g}\colon&
  g' &\subsetneqq h'& &\implies& h&\leq m,\\
&\iff  g \nleq m \land
  \forall h\in \Jirr\setminus\set{g}\colon&
  \gupm[\Mirr]{g}&\subsetneqq \gupm[\Mirr]{h}& &\implies& m& \in\gup{h},\\
g \Hochpfeil m
&\stackrel{I\defeq{\leq}}{\iff}  g \nleq m \land
  \forall a\in \Mirr\setminus\set{m}\colon&
  m' &\subsetneqq a'& &\implies& g&\leq a,\\
&\iff  g \nleq m \land \forall a\in \Mirr\setminus\set{m}\colon&
\mdownj[\Jirr]{m} &\subsetneqq \mdownj[\Jirr]{a}& &\implies& g& \in\mdown{a}.
\end{alignat*}
\end{remark}

The following sufficient condition will lead to our main tool to
establish arrow relations in standard contexts, in particular
in~$\K(\Lp)$.
\begin{lemma}\label{lem:arrows-from-meets-and-joins}
Let $\mathbb{L}=(L,{\leq})$ be any complete lattice. Consider any
$g\in\Jirr$ with unique lower cover~$\tilde{g}\in L$ and
$m\in\Mirr$ with unique upper cover~$\tilde{m}\in L$ such that $g\nleq m$.
Then in the context $\K(\mathbb{L})=(\Jirr,\Mirr,{\leq})$ the following
implications regarding arrow relations hold.
\begin{enumerate}[\upshape(a)]
\item\label{item:meet-g-lcover-below-m}
      If $\tilde{g}\leq m$ and there is a set $S\subseteq\Mirr$ such that
      $g=\bigwedge S$, then $g\Runterpfeil m$.
\item\label{item:join-m-ucover-above-g}
      If $\tilde{m}\geq g$ and there is a set $S\subseteq\Jirr$ such that
      $m=\bigvee S$, then $g\Hochpfeil m$.
\item\label{item:g-doubly-irr-lcover-below-m}
      If $\tilde{g}\leq m$ and $g\in\Mirr$ (thus doubly completely irreducible),
      then $g\Runterpfeil m$.
\item\label{item:m-doubly-irr-ucover-above-g}
      If $\tilde{m}\geq g$ and $m\in\Jirr$ (thus doubly completely irreducible),
      then $g\Hochpfeil m$.
\end{enumerate}
\end{lemma}
\begin{proof}
We only prove statement~\eqref{item:meet-g-lcover-below-m},
for~\eqref{item:join-m-ucover-above-g} is completely dual;
\eqref{item:g-doubly-irr-lcover-below-m} follows by setting
$S\defeq\set{g}$ in~\eqref{item:meet-g-lcover-below-m},
and~\eqref{item:m-doubly-irr-ucover-above-g} by setting
$S\defeq\set{m}$ in~\eqref{item:join-m-ucover-above-g}.
By assumption we have $g\nleq m$.
To show $g\Runterpfeil m$, according to Remark~\ref{rem:char-arrows-std-cxt},
we take any $h\in\Jirr\setminus\set{g}$ and assume
$\gupm[\Mirr]{g}\subsetneqq\gupm[\Mirr]{h}$. From the hypothesis
of~\eqref{item:meet-g-lcover-below-m} we have $g=\bigwedge S$ with
$S\subseteq \Mirr$, thus we infer
$S\subseteq \gupm[\Mirr]{g}\subseteq \gupm[\Mirr]{h}$.
Therefore, $h$~is a common lower bound of the elements of~$S$, hence
$h\leq \bigwedge S = g$. Since $h\neq g$, we have $h<g$, and,
as $g\in\Jirr$, we infer $h\leq \tilde{g}\leq m$, as it is required for
$g\Runterpfeil m$.
\end{proof}

We shall need the following statement for infimum/supremum founded
complete lattices, which can be read
from~\cite[\figurename~1.11, p.~35]{ganter}. For completeness, we
provide a proof of this fundamental fact.
\begin{lemma}\label{lem:foundedness-implies-density}
Let $\La=(L,{\leq})$ be a complete lattice with completely
join-irreducibles~$\Jirr$ and meet-irreducibles~$\Mirr$.
\begin{enumerate}[\upshape(a)]
\item\label{item:inf-founded-implies-M-inf-dense}
      If\/~$\La$ is infimum-founded, then~$\Mirr$ is completely meet-dense.
\item\label{item:sup-founded-implies-J-sup-dense}
      If\/~$\La$ is supremum-founded, then~$\Jirr$ is completely join-dense.
\end{enumerate}
\end{lemma}
\begin{proof}
Part~\eqref{item:sup-founded-implies-J-sup-dense} follows
from~\eqref{item:inf-founded-implies-M-inf-dense} by duality, thus we
only show the latter. Consider any $g\in\La$ and define the subset
$S\defeq \lset{y\in\Mirr}{g\leq y}\subseteq\Mirr$. By its construction,
$S$ satisfies $g\leq \bigwedge S$. Let us assume for a contradiction that
$g<\bigwedge S\eqdef h$. By infimum-foundedness, there is hence a
maximal element $x\in L$ with the properties $g\leq x$ but $h\nleq x$.
If $x\in\Mirr$, then $x\in S$ and thus $h=\bigwedge S\leq x$, being a
contradiction. Therefore, we consider now $x\notin\Mirr$. This means
there must exist a subset $W\subseteq L\setminus\set{x}$ with
$x = \bigwedge W$. It follows for each $w\in W$ that $x\leq w$,
and, in fact, $g\leq x<w$ since $w\neq x$. From the maximality of~$x$
we infer now that $h\nleq w$ fails, i.e., $h\leq w$. Since this holds
for all $w\in W$, we conclude $h\leq \bigwedge W = x$, which is again a
contradiction. Both contradictions show $g=h=\bigwedge S$.
\end{proof}

For our purposes the following characterisation of the arrow relations
is the most appropriate one since the lattice~$\Lp$ is finite, thus all
its chains have only a finite number of elements, and it hence is
doubly founded.

\begin{proposition}\label{prop:arrows-in-doubly-founded-lattices}
Let $\mathbb{L}=(L,{\leq})$ be a doubly founded complete lattice, e.g.,
a finite lattice.
In the formal context $\K(\mathbb{L})=(\Jirr,\Mirr,{\leq})$ the arrow
relations can be characterised as follows (cf.\
\figurename~\ref{fig:arrows-in-doubly-founded-lattices}).
Consider any $g\in\Jirr$ with unique lower
cover~$\tilde{g}\in L$ and $m\in\Mirr$ with unique upper
cover ${\tilde{m}\in L}$. Then we have 
\begin{enumerate}[\upshape(a)]
\item\label{item:g-lcover-below-m}
      $g\Runterpfeil m \iff g\nleq m \land \tilde{g}\leq m$;
\item\label{item:m-ucover-above-g}
      $g\Hochpfeil m \iff g\nleq m \land \tilde{m}\geq g$.
\end{enumerate}
\end{proposition}
\begin{proof}
We shall only prove~\eqref{item:g-lcover-below-m}
since~\eqref{item:m-ucover-above-g} is completely dual.
To show `$\Longleftarrow$' we use infimum-foundedness, which
implies that the set $\Mirr$ is infimum-dense, see
Lemma~\ref{lem:foundedness-implies-density}\eqref{item:inf-founded-implies-M-inf-dense}.
This means that every
$g\in L$ can be written as $\bigwedge S$ for some set $S\subseteq\Mirr$,
namely, we may take $S=\lset{y\in\Mirr}{g\leq y}$. This is true, in
particular, for all $g\in\Jirr$, hence $g\nleq m$, $\tilde{g}\leq m$ and
Lemma~\ref{lem:arrows-from-meets-and-joins}\eqref{item:meet-g-lcover-below-m}
imply $g\Runterpfeil m$.
\par
For the converse implication, let us assume that $g\Runterpfeil m$
holds. This implies $g\nleq m$ by Remark~\ref{rem:char-arrows-std-cxt}.
By infimum-foundedness, there is $p\in L$ that is maximal with
respect to the property $\tilde{g}\leq p$ but $g\nleq p$.
Let $U\defeq\lset{z\in L}{z>p}$. For every $z\in U$ we have
$z>p\geq \tilde{g}$, thus, in order to not violate the maximality
of~$p$, the element~$z$ must fail the property $g\nleq z$, that is,
$g\leq z$ must hold. Therefore, $g$ is a common lower bound for the
elements of~$U$, and hence we have $g\leq \bigwedge U$.
As $g\nleq p$, we know that $\bigwedge U \neq p$, thus in fact,
$\bigwedge U>p$ since all $z\in U$ are above~$p$.
Consequently, if $p=\bigwedge T$ for some subset $T\subseteq L$, then
$p\in T$, for otherwise $T\subseteq U$ and thus
$p<\bigwedge U\leq \bigwedge T$.
This shows that $p\in\Mirr$.
Since~$\mathbb{L}$ is supremum-founded, the set $\Jirr$ is supremum-dense, see
Lemma~\ref{lem:foundedness-implies-density}\eqref{item:sup-founded-implies-J-sup-dense}.
Thus we can write $\tilde{g}=\bigvee S$ for some set
$S\subseteq \Jirr$. For every~$h\in S$ we have
$h\leq \bigvee S =\tilde{g}<g$, wherefore
$\gupm[\Mirr]{g}\subseteq \gupm[\Mirr]{h}$.
As $p\in\Mirr$ and $p\geq \tilde{g}$, we have $p\in\gupm[\Mirr]{h}$,
but certainly $p\notin \gupm[\Mirr]{g}$ since $p\ngeq g$.
Thus, $p\in\gupm[\Mirr]{h}\setminus\gupm[\Mirr]{g}$, i.e.,
$\gupm[\Mirr]{g}\subsetneqq \gupm[\Mirr]{h}$. Now $g\Runterpfeil m$ and
Remark~\ref{rem:char-arrows-std-cxt} imply $h\leq m$. As $h\in S$ was
arbitrary, we conclude that $\tilde{g}=\bigvee S \leq m$.
\end{proof}

\begin{figure}[hbt!]
		\centering
\setlength{\abovedisplayskip}{0pt plus 0pt minus 25pt}
\setlength{\belowdisplayskip}{0pt plus 0pt minus 25pt}
\includegraphics[scale=1]{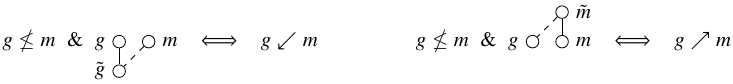}%
\caption{Graphical representation of
Proposition~\ref{prop:arrows-in-doubly-founded-lattices}; solid edges represent the
covering relation.}
\label{fig:arrows-in-doubly-founded-lattices}
\end{figure}

In Section~\ref{sect:intro} we explained that so-called arrow-closed
subcontexts (of the standard context) are a key ingredient in order to
obtain subdirect decompositions of finite lattices. We now provide
concrete definitions as far as they are needed in this paper. A
\emph{subcontext} of a context $\K=(G,M,I)$
is a context $\tilde{\K}=(H,N,J)$
where $H\subseteq G$, $N\subseteq M$ and $J= I\cap (H\times N)$.
For a clarified context $\K$, that is,
$g'=h'$ implies $g=h$, and $m'=n'$ implies $m=n$
for all $g,h\in G$ and $m,n\in M$, such a
subcontext~$\tilde{\K}$ is \emph{arrow-closed} if
for all $h\in H$, $m\in M$, $n\in N$ and $g\in G$
the condition $g\Hochpfeil m$ implies $m\in N$,
and $g\Runterpfeil n$ implies $g\in H$,
see~\cite[Definition~46]{ganter}. Note that the standard context of a
finite lattice~$\La$ is always clarified and
reduced~\cite[Proposition~12]{ganter}.
For a finite clarified context $\K=(G,M,I)$ and $G_1\subseteq G$ and $M_1\subseteq M$
there is always a smallest arrow-closed subcontext
$\tilde{\K}=(H,N,I\cap(H\times N))$ of $\K$ with $G_1\subseteq H$ and
$M_1\subseteq N$. It can be obtained by constructing the directed graph
$\apply{G\disjointunion M, {\Hochpfeil}\cup {(\Runterpfeil)^{-1}}}$ and
considering the (not necessarily strongly) connected directed
components~$[x]$ of each $x\in G_1\disjointunion M_1$. One then forms
$\bigcup_{x\in G_1\disjointunion M_1} [x]$, which can be written in a
unique way as $H\disjointunion N$ with $H\subseteq G$ and $N\subseteq M$.
In particular, starting from
$G_1=\set{g}$ and $M_1=\emptyset$ with $g\in G$ (or dually from
$G_1=\emptyset$ and $M_1=\set{m}$ with $m\in M$), we get the
\emph{one-generated arrow-closed subcontexts} of~$\K$,
cf.~\cite[Section~4.1]{ganter}. Note that if the context~$\K$ is
reduced, we may always concentrate on using either only objects or
attributes for constructing all its one-generated arrow-closed
subcontexts.

\subsection{Integer partitions}
Our aim is to study the arrow relations of the standard context of the
lattice~$\Lp$ of positive integer partitions, which is formed by the
sets~$\Ji$ and~$\Mi$.
First, we define formally, what a partition and the dominance
order is.
\begin{definition}
An (ordered) \emph{partition} of a number $n\in \N$ is an $n$-tuple $a\defeq\tup{a}{n}$ of natural numbers such that
\[a_1\geq a_2\geq \dotsm \geq a_n\geq 0 \qquad \text{  and  } \qquad
n=a_1+a_2+\dotsm+a_n.\]
If there is $k\in \set{1,\dotsc,n}$ such that $a_k\geq1$ and $a_i=0$ for
all $i>k$, we also allow for the partition~$a$ to be written in the form
$\apply{a_1, a_2, \dotsc, a_k}$, where we have deleted the zeros at the end.
\end{definition}

For example, $(5,4,1,1,0,0,0,0,0,0,0)$ is a partition of~$11$ because
${5\geq 4\geq 1\geq 0}$ and $5+4+1+1=11$. By removing the
trailing zeros we can represent it more compactly as $(5,4,1,1)$.
Graphically, we can illustrate a partition
using a diagram drawn with small squares or `bricks' arranged in a
downward ladder shape (cf.\ \figurename~\ref{fig:Ferrers-diagrams}), which
is known as \emph{Ferrers diagram} (usually Ferrers diagrams are drawn rotated
clockwise by a $90$~degrees angle, but the chosen presentation is more
useful to us, cf.\ Definition~\ref{def:transition-rules-lower-covers}).
Given a partition~$a$ of~$n$, one obtains
the \emph{conjugated} or \emph{dual partition} $a^*$ in the sense
of~\cite{Brylawski} and~\cite{ganter2022notes} as
$a^* = (a_1^*,\dotsc,a_n^*)$ where
$a_i^*\defeq\crd{\{1\leq j\leq n\mid a_j \geq i\}}$
for all $i\in \set{1,\dotsc,n}$. The Ferrers diagram of~$a^*$ can be
seen from the diagram of~$a$ by reading it by rows, from bottom to top.
For instance, the partition $(5,4,1,1)$ of $11=5+4+1+1$ has
the Ferrers diagram shown in \figurename~\ref{fig:Ferrers-diagrams},
and its conjugate consists of~$4$
bricks from the first row, $2$~bricks from the second to fourth, and $1$~brick
from the fifth row. We thereby obtain the partition
$(4,2,2,2,1)$; its Ferrers diagram is also depicted in
\figurename~\ref{fig:Ferrers-diagrams}.
\begin{figure}[ht]
\centering
\includegraphics[scale=1]{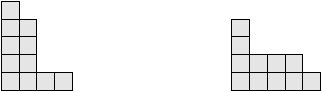}%
\caption{Ferrers diagrams. Left: of $g=(5,4,1,1)$; right: of the
conjugated (dual) partition $g^*=(4,2,2,2,1)$.}
\label{fig:Ferrers-diagrams}
\end{figure}

We denote the \emph{set of all partitions} of an integer $n\in \N$ by
$\Part(n)$.
From the construction of the conjugate via the Ferrers diagram, it is
easy to see that $\apply{a^{*}}^{*} = a$ holds for every $a\in\Part(n)$.
Therefore, partition conjugation is an involutive permutation
of~$\Part(n)$. One may order the set~$\Part(n)$ in different ways, for
example, lexicographically, or pointwise. In this paper we are
interested in the dominance order established by
Brylawski~\cite{Brylawski}.

\begin{definition}[\cite{Brylawski}]\label{def:dominance-order}
Let $a=\tup{a}{n}$ and $b=\tup{b}{n}$ be two partitions of $n\in\N$.
We define the \emph{dominance order} between~$a$ and~$b$ by setting
$a\geq b$ if and only if $\sum_{i=1}^j a_i\geq\sum_{i=1}^j b_i$
holds for all $j\in\set{1,\dotsc,n}$.
\end{definition}
Brylawski showed in~\cite[Proposition~2.2]{Brylawski} that the
set~$\Part(n)$ forms a lattice under the dominance order; further
arguments for this were given in~\cite[Chapter~3]{Alain}
and~\cite[Section~2]{ganter2022notes}.
We denote by $\Lp$ \emph{the lattice of all partitions} of $n\in\N$
with the dominance order.
It follows from~\cite[Proposition~2.8]{Brylawski} that for partitions
$a,b\in\Lp$ we have $a\leq b$ if and only if $a^*\geq b^*$, that is,
partition conjugation is an order-antiautomorphism of~$\Lp$,
making~$\Lp$ is self-dual.

Brylawski in~\cite[Proposition~2.3]{Brylawski} characterised precisely
two possibilities for downward movement along the covering
relation of~$\Lp$.
These were later given a more intuitive geometric interpretation as
\emph{transition rules} by
Latapy and Phan~\cite[\figurename~2, p.~1358]{LATAPY},
which we are going to follow in this article. The subsequent
definitions are required for this.

\begin{definition}[{cf.~\cite[p.~1358]{LATAPY}}]
Let $n,j\in\N$ and $1\leq j<n$.
The partition $a=\tup{a}{n}\in\Lp$ has
\begin{enumerate}[(1)]
 \item a \emph{cliff} at~$j$\/ if $a_j-a_{j+1}\geq 2$;
 \item a \emph{slippery step} at~$j$ if there is $\ell\in\N$
       such that $2\leq \ell\leq n-j$ and
       $a_{j}-1=a_{j+1}=\dots=a_{j+\ell-1}=a_{j+\ell}+1$.
\end{enumerate}
\end{definition}
For example, the left Ferrers diagram in
\figurename~\ref{fig:Ferrers-diagrams} has a cliff in the second
position ($j=2$) and its dual has a cliff at $j=1$ and a slippery step
(with $\ell=2$) at $j=4$.

\begin{definition}[{Transition rules, cf.~\cite{Brylawski,LATAPY}}]%
\label{def:transition-rules-lower-covers}
Let $n\in\N$ and $a\in\Lp$.
\begin{enumerate}[\upshape (1)]
\item\label{item:transition-step}
      If $a = (\dots, k+1,k,\dotsc,k,k-1,\dots)$ for some $k\in\N$ with
      $1\leq k<n$ has a slippery step, then the brick at the slippery
      step may slip across the step to give
      $\tilde{a}=(\dotsc,k,k,\dotsc,k,k,\dotsc)$.
      The subsequent illustration shows the application of such a
      transition to a slippery step at position~$j$ (with $i\geq j+2$).
\begin{center}
\includegraphics[scale=1]{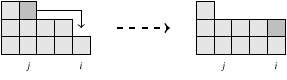}%
\end{center}

\item\label{item:transition-cliff}
      If $a= (\dotsc,k,k-h,\dotsc)$ for some $k,h\in\N$ with
      $2\leq h\leq k\leq n$
      has a cliff, then the brick may fall from the cliff to give
      $\tilde{a}=(\dotsc,k-1,k-h+1,\dotsc)$.
      Again, this is illustrated with a cliff of height~$3$ at position~$j$.
\begin{center}
\includegraphics[scale=1]{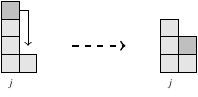}%
\end{center}
\end{enumerate}
\end{definition}

\begin{lemma}[{\cite[Proposition~2.3]{Brylawski} and~\cite{LATAPY}}]%
\label{lem:description-of-lower-covers}
The set of lower covers of a partition $g\in\Lp$ consists exactly of
all partitions~$\tilde{g}$ that can be obtained from~$g$ by applying
one of the two transition rules~\eqref{item:transition-step}
or~\eqref{item:transition-cliff} described in
Definition~\ref{def:transition-rules-lower-covers}.
\end{lemma}

Since the completely join-irreducible elements of~$\Lp$ have exactly
one lower cover, according to
Lemma~\ref{lem:description-of-lower-covers}, the partitions in~$\Ji$
are precisely those to which exactly one of the transition
rules~\eqref{item:transition-step} or~\eqref{item:transition-cliff}
applies. Therefore, the $\bigvee$-irreducible partitions of~$n$ can
be characterised as those that have exactly one cliff (and no slippery
step) or exactly one slippery step (and no cliff). Based on this
Brylawski~\cite{Brylawski} split the set~$\Ji$ into four categories,
extending this by conjugation to~$\Mi$.
To be able to express our results more compactly, in the following
statement we have slightly modified the borders between the different
types of irreducibles compared to~\cite{Brylawski}, making them in
particular disjoint.

\begin{lemma}[{cf.~\cite[Corollary 2.5]{Brylawski}}]\label{lem:char-join-irr}
For $n \in\N$ the completely join-irreducible partitions
of~$\Ji$ can be categorised into four types where $b\in\N$, $d,\ell\in\Nplus$:
\begin{enumerate}[\upshape\bfseries Type A:]
\item\label{typeA} $\tupone{k}{\ell}$ for $k\geq 2$.
\item\label{typeB} $\tupdos{k}{b}{k-1}{b+\ell}$ for $k\geq 2$, $b\geq 1$.
\item\label{typeC} $(k,1,1,\dotsc,\overset{d+1}{1})$ for $k\geq 3$.
\item\label{typeD}
$(k+1,k+1,\dotsc,\overset{b}{k+1},k,\dotsc,\overset{b+\ell}{k},1,1,\dotsc,\overset{b+l+d}{1})$
for $k\geq 3$, $b+\ell\geq 2$.
\end{enumerate}
Also the $\bigwedge$-irreducible elements can be split into four groups
where $c\in\Nplus$:
\begin{enumerate}[\upshape\bfseries Type I:]
\item\label{typeI}   $\tupone{t}{c}$
                     for $t\geq 1$, $c\geq 2$,
                     i.e., $t$ appears at least twice.
\item\label{typeII}  $(t,t,\dotsc,\overset{c\vphantom{1}}{t},r)$
                     for $t>r\geq 1$.
\item\label{typeIII} $(a,1,1,\dotsc,\overset{c+1}{1})$
                     for $a\geq 2$, $c\geq 2$,
                     i.e., there are at least two $1$s.
\item\label{typeIV}  $(a,t,t,\dotsc,\overset{c+1}{t},r)$
                     for $a>t>r\geq 0$, $t,c\geq 2$,
                     i.e., $t\geq2$ appears at least twice.
\end{enumerate}
\end{lemma}
Observe that for each
$(L,J)\in\set{(\ref{typeA},\ref{typeI}),
              (\ref{typeB},\ref{typeII}),
              (\ref{typeC},\ref{typeIII}),
              (\ref{typeD},\ref{typeIV})}$
we have that if $g\in\Ji$ has type~$L$, then $g^*$ has type~$J$,
and if $m\in\Mi$ has type~$J$, then $m^*$ has type~$L$.
Therefore, the pairs $(L,J)$ of categories of completely irreducible
elements of~$\Lp$ are completely dual to each other.
\par

We mentioned that~$\Lp$ is a self-dual lattice under partition conjugation
as involutive antiautomorphism, see~\cite[Proposition~2.8]{Brylawski}.
Using this fact we obtain the following simple but useful results to
switch down and up-arrows.

\begin{lemma}\label{Lemma_3.4}
Let $g\in \Ji$ and $m\in \Mi$ with duals~$g^{*}$ and~$m^{*}$,
respectively. Then it holds that
\[g \Runterpfeil m \iff m^* \Hochpfeil g^*.\]
\end{lemma}
\begin{proof}
The proof is a routine calculation exploiting the involutive
antiautomorphism ${}^{*}\colon\Lp\to\Lp$ and the duality of the
involved concepts.
\end{proof}

Similarly, one can prove the following lemma.
\begin{lemma}\label{Lemma_3.5}
Let $g\in \Ji$ and $m\in \Mi$ with duals~$g^{*}$ and~$m^{*}$,
respectively. Then it holds that
\[g \Hochpfeil m \iff m^* \Runterpfeil g^*.\]
\end{lemma}

Combining Lemma~\ref{Lemma_3.4} and Lemma~\ref{Lemma_3.5}, we obtain
the following corollary.
\begin{corollary}\label{Corollary_3.6}
Let $g\in \Ji$ and $m\in \Mi$ with duals $g^{*}\in\Mi$ and
$m^{*}\in\Ji$. Then it holds that
\[g \Doppelpfeil m \iff m^* \Doppelpfeil g^*.\]
\end{corollary}

To obtain some familiarity with the completely irreducible elements
of~$\Lp$ and the characterisations concerning the arrow relations given
in Subsection~\ref{subsect:fca}, we consider the lattice~$\Lp[7]$ shown
in \figurename~\ref{fig:lattice7}.

\begin{figure}[thb]
\centering
	\includegraphics[scale=1]{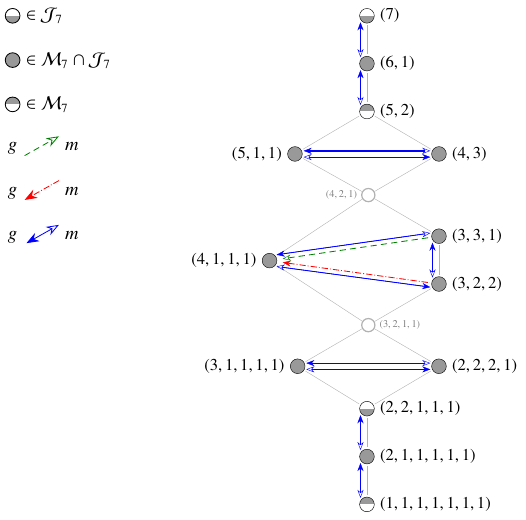}%
\caption{Lattice~$\Lp[7]$ with all arrows.
Up-arrows are shown with open arrow tips ending in
$\bigwedge$-irreducibles, and down-arrows with filled arrow tips ending
in $\bigvee$-irreducibles;
single up-arrows are coloured green, single down-arrows red.}
\label{fig:lattice7}
\end{figure}
\begin{example}\label{ex:lattice7}
  For $n=7$ we have from Lemma~\ref{lem:char-join-irr}, cf.\
  also~\cite[Proposition~2]{ganter2022notes}, that
	\[
		\begin{aligned}
    \Ji[7]&= \set{(2,1,\dotsc,1),(2,2,1,1,1),(2,2,2,1),(3,1,\dotsc,1),(3,2,2),(3,3,1),(4,1,1,1),(4,3),(5,1,1),(6,1), (7)},\\
    \Mi[7]&= \set{(1,\dotsc,1),(2,1,\dotsc,1),(2,2,2,1),(3,1,\dotsc,1),(3,2,2),(3,3,1),(4,1,1,1),(4,3),(5,1,1),(5,2),(6,1)}.
		\end{aligned}
	\]
The standard context $\K(\Lp[7])=\apply{\Ji[7],\Mi[7],\leq}$ is
presented in \figurename~\ref{fig:std-cxt-lattice7}. In both
\figurename~\ref{fig:std-cxt-lattice7} and
\figurename~\ref{fig:lattice7} we have indicated the arrow relations.
Note that $n=7$ is the first case where single arrows appear. In fact
there is exactly one single up-arrow and one single down-arrow
in~$\K(\Lp[7])$ (where the respective opposite arrow is not present).
We shall investigate this in the case of the down-arrow
$(4,1,1,1)\mathrel{\color{red}\Runterpfeil}(3,2,2)$, shown in red in
\figurename~\ref{fig:lattice7}. The green up-arrow can then be
explained by self-duality of~$\Lp[7]$ using Lemma~\ref{Lemma_3.5}.
\par
The partition $g\defeq(4,1,1,1)$ is completely join-irreducible of
type~\ref{typeC}, while $m=(3,2,2)$ is completely meet-irreducible of
type~\ref{typeIV}. To explain the down-arrow we shall use
Proposition~\ref{prop:arrows-in-doubly-founded-lattices}\eqref{item:g-lcover-below-m}.
Clearly, $g\nleq m$ since the condition on the partial sums already
fails with $4\nleq 3$ in the first position. The partition~$g$ has a
single cliff in the first position, from which a brick may fall to
obtain the unique lower cover $\tilde{g}=(3,2,1,1)$, cf.\ transition
rule~\eqref{item:transition-cliff} and \figurename~\ref{fig:lattice7}.
Indeed, the lattice diagram confirms that $\tilde{g}\leq m$, which can
also be checked with Definition~\ref{def:dominance-order}.
Thus, by
Proposition~\ref{prop:arrows-in-doubly-founded-lattices}\eqref{item:g-lcover-below-m}
we conclude $g\mathrel{\color{red}\Runterpfeil} m$. Moreover, that this
arrow is not a double arrow, can also be explained with the help of
Proposition~\ref{prop:arrows-in-doubly-founded-lattices}.
Namely, the unique upper cover of~$m$ is $\tilde{m}=(3,3,1)$ since we
may let the brick fall from the cliff in the second position
of~$\tilde{m}$ to get~$m$, cf.\ transition
rule~\eqref{item:transition-cliff}. Since $g\nleq\tilde{m}$, we have
$g\mathrel{\color{green!50!black}\nHochpfeil} m$ by
Proposition~\ref{prop:arrows-in-doubly-founded-lattices}\eqref{item:m-ucover-above-g}.
\par
The fact that there is a \emph{unique} up-arrow (and a unique
down-arrow) in~$\K(\Lp[7])$ that fails to be a double arrow, and why
non-double arrows appear for $n=7$ for the first time, cannot be
explained yet. This requires the characterisations in
Sections~\ref{sect:double-arrows} and~\ref{sect:single-arrows}.

\begin{figure}[hbt]
		\centering
\begin{tabular}{|l||c|c|c|c|c|c|c|c|c|c|c|}\hline
$\K\apply{\Lp[7]}$ & 1111111 & 211111 &31111 &2221 &4111 &322 &331 &511 &43 &52 &61\\\hline\hline
7& & & & & & & & & & &$\DoubleArrow$\\ \hline
61& & & & & & & & & &$\DoubleArrow$&$\times$\\ \hline
511& & & & & & & &$\times$&$\DoubleArrow$&$\times$&$\times$\\ \hline
43& & & & & & & &$\DoubleArrow$&$\times$&$\times$&$\times$\\ \hline
4111& & & & &$\times$&$\color{red}\DownArrow$&$\DoubleArrow$&$\times$&$\times$&$\times$&$\times$\\ \hline
331& & & & &$\color{green!50!black}\UpArrow$&$\DoubleArrow$&$\times$&$\times$&$\times$&$\times$&$\times$\\ \hline
322& & & & &$\DoubleArrow$&$\times$&$\times$&$\times$&$\times$&$\times$&$\times$\\ \hline
31111& & &$\times$&$\DoubleArrow$&$\times$&$\times$&$\times$&$\times$&$\times$&$\times$&$\times$\\ \hline
2221& & &$\DoubleArrow$&$\times$&$\times$&$\times$&$\times$&$\times$&$\times$&$\times$&$\times$\\ \hline
22111& &$\DoubleArrow$&$\times$&$\times$&$\times$&$\times$&$\times$&$\times$&$\times$&$\times$&$\times$\\ \hline
211111&$\DoubleArrow$&$\times$&$\times$&$\times$&$\times$&$\times$&$\times$&$\times$&$\times$&$\times$&$\times$\\ \hline
\end{tabular}
\caption{Standard context~$\K\apply{\Lp[7]}=(\Ji[7],\Mi[7],{\leq})$ with arrow relations shown.}
\label{fig:std-cxt-lattice7}
\end{figure}

\end{example}

\section{General facts concerning the dominance order of integer partitions}
\label{sect:general-facts-dominance}
We start with a lemma giving a simple sufficient condition for the
dominance order among partitions.
\begin{lemma}\label{lem:dominance-from-complement}
Let $n,j,k,m\in\N$ be such that $j,k,m\leq n$ and let
$a=(a_1,\dotsc,a_k,0,\dots), b=(b_1,\dotsc,b_m,0,\dots)\in\Lp$.
If $s\defeq\sum_{i=1}^j a_i = \sum_{i=1}^j b_i$ and
$a_i\geq b_i$ for each $i\in\N$ with $2\leq i\leq j$, while $a_i\leq b_i$
for each $i\in\N$ with $j<i<m$, then $a\leq b$.
\end{lemma}
\begin{proof}
Let $l\in\set{1,\dotsc,n}$. If $l\leq j$, then we have
$\sum_{i=1}^l a_i = s- \sum_{i=l+1}^j a_i \leq
s - \sum_{i=l+1}^j b_i =\sum_{i=1}^l b_i$ since
$b_i\leq a_i$ for $2\leq l+1\leq i\leq j$.
Let now $l>j$. If $l<m$, we have
$\sum_{i=1}^l a_i = s + \sum_{i=j+1}^l a_i
\leq s + \sum_{i=j+1}^l b_i = \sum_{i=1}^l b_i$
since $a_i\leq b_i$ for $j< i\leq l<m$.
Otherwise, $m\leq l\leq n$, and then we have
$\sum_{i=1}^l a_i \leq n = \sum_{i=1}^m b_i = \sum_{i=1}^l b_i$
as $a,b\in\Lp$.
\end{proof}

As an easy corollary we have the situation where~$b$ is not longer
than~$a$ and lies pointwise above~$a$.
\begin{corollary}\label{cor:below-implies-dominance}
Let $n,k,m\in\N$ and let
$a=(a_1,\dotsc,a_k),b=(b_1,\dotsc,b_m)\in\Lp$.
If $a_i\leq b_i$ for all $1\leq i< m$, then $a\leq b$.
\end{corollary}
\begin{proof}
This follows from Lemma~\ref{lem:dominance-from-complement} for
$j=0$ and $s=\sum_{i\in\emptyset} a_i =0=\sum_{i\in\emptyset} b_i$.
\end{proof}

We may also have the somewhat opposite situation.
\begin{corollary}\label{cor:above-implies-dominance}
Let $n,k,m\in\N$ and let
$a=(a_1,\dotsc,a_k),b=(b_1,\dotsc,b_m)\in\Lp$.
If $a_i\geq b_i$ for all $2\leq i\leq m$, then $a\leq b$.
\end{corollary}
\begin{proof}
For $2\leq i\leq m$ we have $a_i\geq b_i$; for $m<i\leq n$ we have
$a_i\geq 0=b_i$. Hence $a_i\geq b_i$ holds for all $2\leq i\leq n$.
Thus we can apply Lemma~\ref{lem:dominance-from-complement} with
$j=n$ and $s=\sum_{i=1}^n a_i =n=\sum_{i=1} b_i$ to get $a\leq b$.
\end{proof}

We now introduce two natural geometric parameters of partitions.
\begin{definition}\label{def:height-width=length}
Let $n\in\N$ and $p=(p_1,p_2,\dots)\in\Lp$.
\begin{enumerate}[\upshape(a)]
\item The \emph{height} of~$p$ is the value of its first entry, that
      is, $\hgt(p) \defeq p_1$.
\item For $h\in\N$ we define
      $\Bh{h}\defeq\lset{p\in\Lp}{\hgt(p)\leq h}$
      to be the set of partitions of height at most~$h$.
\item The \emph{length}, synonymously \emph{width}, of~$p$ is the value
      of the first index $j\in\N$ such that $p_{j+1}=0$, that is to say,
      $\len(p)\defeq\min\lset{j\in\N}{p_{j+1}=0}$.
\item For $w\in\N$ we define
      $\Bw{w}\defeq\lset{p\in\Lp}{\len(p)\leq w}$
      to be the set of partitions of length (width) at most~$w$.
\end{enumerate}
\end{definition}

\begin{lemma}\label{lem:width-in-terms-of-nonzero-entries}
For $n\in\Nplus$ the length of $p\in\Lp$ is the value
of the largest index $j\geq 1$ where $p_j\geq1$, that is,
$\len(p)=\max\lset{j\in\Nplus}{p_j>0}$.
\end{lemma}
\begin{proof}
For $n\in\Nplus$ we have $p_1=\hgt(p)\geq 1$, thus the maximum
$w\defeq\max\lset{j\in\Nplus}{p_j>0}$ exists.
We have $p_w>0$, but, by maximality, $p_{w+1}=0$. By antitonicity
of~$p$ we have $p_j\geq p_w>0$ for $1\leq j\leq w$, thus~$w$ is the
first index where $p_{w+1}=0$. Hence $\len(p)=w$.
\end{proof}

\begin{lemma}\label{lem:char-bounded-width}
For $n,w\in\N$ we have
$\Bw{w}=\lset{p\in\Lp}{p_{w+1}=0}$.
\end{lemma}
\begin{proof}
Let $p\in\Lp$ and $w\in\N$. If $p_{w+1}=0$, then the minimality in the
definition of length implies $\len(p)\leq w$, thus $p\in\Bw{w}$.
Conversely, if $j\defeq \len(p)\leq w$, then $0 = p_{j+1}\geq
p_{w+1}\geq 0$ by the antitonicity of~$p$.
\end{proof}

For parameter values at least~$n$ the concept of bounded height or
width trivialises.
\begin{lemma}\label{lem:bounded-height-width-trivial}
For $n,h\in\N$ such that $h\geq n$, we have $\Bh{h} = \Bw{h} = \Lp$.
\end{lemma}
\begin{proof}
For every $p\in\Lp$ we have $\hgt(p)=p_1\leq n\leq h$, thus $p\in\Bh{h}$.
If $p_{h+1}>0$, then for every $1\leq j\leq h+1$ we would have
$p_j\geq p_{h+1}\geq 1$, thus
$n\geq \sum_{j=1}^{h+1} p_j \geq\sum_{j=1}^{h+1} 1 = h+1\geq n+1$,
which is a contradiction. Therefore, $p_{h+1}=0$ and $p\in\Bw{h}$ by
Lemma~\ref{lem:char-bounded-width}.
\end{proof}

In a similar fashion, for $n\in\Nplus$, we have
$\Bh{0} = \Bw{0} = \emptyset$ since $p_1=0$ implies $p=(0,0\dots)$ and
thus $n=0$. Therefore, for $n\in\Nplus$ we may
safely restrict our attention to parameter values $1\leq h\leq n$ to
bound the height or width of a partition, which is, of course, clear
from the geometric intuition via the Ferrers diagrams.
\par
It is useful to observe that height and width are dual concepts
with respect to partition conjugation.
\begin{lemma}\label{lem:width-height-dual}
For $n\in\N$ and $p\in\Lp$ we have
$\hgt(p^{*}) = \len(p)$ and $\len(p^{*})= \hgt(p)$.
Moreover, for $w\in\N$ we have $p\in\Bw{w}$ exactly if $p^*\in\Bh{w}$.
\end{lemma}
\begin{proof}
Let $w\defeq \len(p)$ such that $p=(p_1,\dotsc,p_w)$ with $p_j\geq 1$
for $1\leq j\leq w$. From the definition of the conjugate we infer that
the first entry of~$p^*$ is~$w$, hence $\hgt(p^*)=w = \len(p)$. The
second equality follows from $\hgt(p^*)= \len(p)$ since ${p^{*}}^* = p$.
Finally, for $w\in\N$ we have $p\in \Bw{w}$ if and only if
$\hgt(p^*) = \len(p) \leq w$, i.e., if and only if $p^*\in\Bh{w}$.
\end{proof}

We may also observe a basic necessary condition following from the
dominance order: the dominated partition is always at least as long as
the dominating one.
\begin{lemma}\label{lem:shorter-implies-nleq}
Let $n\in\N$ and $a,b\in\Lp$.
If\/ $\len(a)<\len(b)$, then $a\nleq b$.
\end{lemma}
\begin{proof}
Assume $l\defeq \len(a)<w\defeq \len(b)$.
Then $a_{l+1}=0$ and $n= \sum_{i=1}^l a_i$.
Since $l<w$ we have $l+1\leq w$, whence
$b_{l+1} \geq b_w\geq 1$.
Thus $\sum_{i=1}^l b_i<\sum_{i=1}^{l+1}b_i\leq n=\sum_{i=1}^l a_i$,
demonstrating $a\nleq b$.
\end{proof}

The following lemma describes the largest partition of a given least
length~$\ell$.
\begin{lemma}\label{lem:top-typeIII-least-length}
Let $n,\ell\in\N$ with $n\geq \ell\geq 1$, let $p\in\Lp$ and set
$m\defeq (n-(\ell-1),1,\dotsc,\overset{\ell}{1})$.
If\/ $\len(p)\geq\ell$, then $p\leq m$;
if\/ $\len(p)<\ell$, then $p\nleq m$.
\end{lemma}
\begin{proof}
Take $p\in\Lp$ with  $\len(p)\ge\ell$, that is,
$p=(p_1,\dotsc,p_\ell,\ldots)$ with
$p_i\geq p_{\ell}\geq 1$ for $1\leq i\leq \ell$.
Thus, for $2\leq i\leq\ell$ we have $m_i=1\leq p_i$,
hence $p\leq m$ by
Corollary~\ref{cor:above-implies-dominance}.
Furthermore, if~$p$ is shorter than~$m$, then
$p\nleq m$ by Lemma~\ref{lem:shorter-implies-nleq}.
\end{proof}

There is also a second largest partition of
a given least length $\ell$.
\begin{corollary}\label{cor:subtop-typeIII-least-length}
Let $n,\ell\in\Nplus$ with $2\leq \ell, n-\ell$.
Then
$m=(n-(\ell-1),1,\dotsc,\overset{\ell}{1})\in\Ji$ is of
type~\ref{typeC} and has
$\tilde{m} = (n-\ell,2,1,\dotsc,\overset{\ell}{1})$ as its
unique lower cover (the part $1,\dotsc,1$ appearing as of $\ell\geq3$);
moreover any $p\in\Lp\setminus\set{m}$ with
$\len(p)\geq\ell$ satisfies $p\leq \tilde{m}$.
\end{corollary}
\begin{proof}
Since $n-(\ell-1)\geq 3$ and $\ell\geq 2$, we see that~$m$ is of
type~\ref{typeC}, and we can obtain the unique lower cover~$\tilde{m}$
by applying transition rule~\eqref{item:transition-cliff} to the cliff
in the first position of~$m$.
If $p\in\Lp$ has length~$\len(p)\ge\ell$, then
$p\leq m$ by Lemma~\ref{lem:top-typeIII-least-length}.
If we additionally assume
$p\neq m$, then $p<m$ and thus $p\leq \tilde{m}$,
for~$\tilde{m}$ is the unique lower cover of~$m$.
\end{proof}

We can also identify a largest partition with a given height bound.
\begin{lemma}\label{lem:largest-bounded-height}
Let $n,h\in\Nplus$ with $h\leq n$ and factorise $n = wh+r$ where
$r,w\in\N$, $0\leq r<h$. Then $w\geq 1$ and every $p\in\Bh{h}$ satisfies
$p\leq m\defeq (h,\dotsc,\overset{w}{h},r)\in\Bh{h}$.
\end{lemma}
\begin{proof}
Since $n=wh+r$, we have $m\in\Lp$. If $w=0$, then $n=r<h\leq n$; thus
$w\geq 1$ and therefore $\hgt(m)=h$, i.e., $m\in\Bh{h}$. By
antitonicity, for each $1\leq i\leq w$ we have
$p_i\leq p_1=\hgt(p)\leq h = m_i$,
hence $p\leq m$ by Corollary~\ref{cor:below-implies-dominance}.
\end{proof}

By duality we have a least partition of a given bounded width.
\begin{lemma}\label{lem:least-bounded-width}
Let $n,w\in\Nplus$, $w\leq n$ and factorise $n=\kappa w+b$ where
$\kappa,b\in\N$, $0\leq b<w$.
Then $\kappa\geq 1$ and every $p\in\Bw{w}$ satisfies
$p\geq
g=(\kappa+1,\dotsc,\overset{b}{\kappa+1},\kappa,\dotsc,\overset{w}{\kappa})$.
\end{lemma}
\begin{proof}
By Lemma~\ref{lem:width-height-dual}, for $p\in\Lp$ we have
$p\in\Bw{w}$ if and only if $p^*\in\Bh{w}$. Therefore, by
Lemma~\ref{lem:largest-bounded-height}, we obtain
$p^*\leq m = (w,\dotsc,\overset{\kappa}{w},b)$, and
since conjugation is an involutive antiautomorphism of~$\Lp$,
this inequality implies
$p={p^*}^*\geq m^* =
(\kappa+1,\dotsc,\overset{b}{\kappa+1},\kappa,\dotsc,\overset{w}{\kappa})=g$.
\end{proof}

\begin{corollary}\label{cor:least-bounded-width-typeAB}
Let $b\in\N$, $k,\ell\in\Nplus$ and set $n\defeq b(k+1)+\ell k$.
The least partition in $\Bw{b+\ell}$
has the form $g=(k+1,\dotsc,\overset{b}{k+1},k,\dotsc,\overset{b+\ell}{k})$.
If $b+\ell<n$, then $g\in\Ji$ is of type~\ref{typeA} and $k\geq 2$,
or~$g$ is of type~\ref{typeB}.
\end{corollary}
\begin{proof}
We set $w\defeq b+\ell\geq 1$.
Since $k\geq 1$ we have
$n = b(k+1)+\ell k\geq b+\ell=w$, moreover
$n = b+(b+\ell) k = b+wk$.
Since $0\le b<b+\ell=w$, this provides the factorisation of~$n$
modulo~$w$,
whence Lemma~\ref{lem:least-bounded-width} shows that the least element
of $\Bw{w}$ is
$g = (k+1,\dotsc,\overset{b}{k+1},k,\dotsc,\overset{b+\ell}{k})$.
Furthermore, if $b\geq 1$, then $g\in\Ji$ is of type~\ref{typeB}.
If otherwise $b=0$, then $\ell=b+\ell<n=\ell k$ implies $k>1$, i.e.,
$k\geq 2$. Therefore, $g=(k,\dotsc,\overset{\ell}{k})$ is of
type~\ref{typeA}.
\end{proof}

\section[Description of the double arrow relation in
the standard context of partitions of n]%
{Description of the double arrow relation in
the standard context $\apply{\Ji, \Mi,\leq}$.}
\label{sect:double-arrows}

It is our aim to describe the arrow relations of the standard
context~$\K\apply{\Lp}$. To do this, we start by giving a complete
characterisation of all arrow relations of $\K\apply{\Lp[7]}$. To this
end we present the following lemmata, which are true for $\Lp$ in
general. The first lemma is quite elementary, but useful, as it
describes the top part of~$\Lp$.
\begin{lemma}\label{lem:topchain}
Let $n\in\N$ such that $n\geq 2$.
\begin{enumerate}[\upshape(a)]
\item\label{item:topchain-2}
  Then $(n)\in\Ji$ and $(n-1,1)\in \Mi$ is its unique
  lower cover.
\item\label{item:topchain-3}
  If $n\geq 3$, then $(n-1,1)\in \Mi\cap\Ji$.
\item\label{item:topchain-4}
  If $n\geq 4$, then $(n-2,2)\in\Mi$ is the unique lower cover of
  $(n-1,1)$.  Hence the top part of~$\Lp$ has the shape depicted in
  \figurename~\ref{fig:chain}.
\end{enumerate}
\end{lemma}
\begin{figure}[h]
\centering
\includegraphics[scale=1]{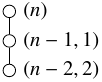}%
\caption{Chain at the top of~$\Lp$ for $n\geq 4$;
         edges denote the covering relation.}
\label{fig:chain}
\end{figure}
\begin{proof}
The top partition $(n)$ is completely
join-irreducible, since $n\geq 2$ ensures that it is of type~\ref{typeA}.
Letting a brick fall from the cliff, we obtain via transition
rule~\eqref{item:transition-cliff} its unique lower cover $(n-1,1)$.
The latter partition is completely meet-irreducible, for it is of
type~\ref{typeI} when $n=2$ and of type~\ref{typeII} if $n\geq 3$. This shows
part~\eqref{item:topchain-2}.
\par
Let us now assume that $n\geq 3$. If $n=3$, then $(n-1,1)$ is
completely join-irreducible of type~\ref{typeB}; moreover, it is of
type~\ref{typeC} when $n\geq 4$. Part~\eqref{item:topchain-3} has been
demonstrated.
\par
Finally let us impose $n\geq 4$. By applying transition
rule~\eqref{item:transition-cliff}, that is, by letting a brick fall
from the cliff in the first position of $(n-1,1)$, we obtain $(n-2,2)$
as its lower cover, which is unique by~\eqref{item:topchain-3}.
Moreover, $(n-2,2)$ is completely meet-irreducible of type~\ref{typeI} if $n=4$
and of type~\ref{typeII} if $n\geq 5$.
\end{proof}

\begin{lemma} \label{First_Lemma}
 For $n\geq 2$, we have $g=(n) \in \Ji$, $m=\apply{n-1,1}\in \Mi$, and
 $g \Doppelpfeil m$ in~$\K(\Lp)$.
\end{lemma}
\begin{proof}
From Lemma~\ref{lem:topchain}\eqref{item:topchain-2}, we have
$g=(n) \in \Ji$, $m=\apply{n-1,1}\in \Mi$, and clearly
$\tilde{g}=m$ is the unique lower cover of~$g$ and $\tilde{m}=g$ is the
unique upper cover of~$m$. Since $g\nleq m$,
Proposition~\ref{prop:arrows-in-doubly-founded-lattices} directly
implies $g\Doppelpfeil m$.
\end{proof}

\begin{lemma}\label{Second_Lemma}
 For $n\geq 4$ we have $g=\apply{n-1,1} \in \Ji$,
 $m=\apply{n-2,2}\in \Mi$, and $g \Doppelpfeil m$ in~$\K(\Lp)$.
\end{lemma}
\begin{proof}
As $n\geq 4$, we know from Lemma~\ref{lem:topchain} that
$g \in \Ji\cap\Mi$, having $\tilde{g}\defeq m\in \Mi$ as its unique
lower cover. At the same time $\tilde{m}=g$ is thus the (unique) upper
cover of~$m$. Since $g\nleq m$,
Proposition~\ref{prop:arrows-in-doubly-founded-lattices}
directly yields $g\Doppelpfeil m$.
\end{proof}

Applying Corollary~\ref{Corollary_3.6} to Lemmata~\ref{First_Lemma}
and~\ref{Second_Lemma}, respectively, we obtain two additional double
arrows in $\K(\Lp)$.
\begin{lemma}\label{Lemma3.7}
In the standard context~$\K(\Lp)$ of~$\Lp$ the following facts hold:
\begin{enumerate}[\upshape(a)]
\item\label{item:double-arrow-n-n-1-dual}
      For $n\geq 2$ we have
      $g=(2,1,\dotsc,\overset{n-1}{1})\in\Ji$,
      $m=(1,1,\dotsc,\overset{n}{1})\in\Mi$, and
      $g \Doppelpfeil m$.
\item\label{item:double-arrow-n-1-n-2-dual}
      For $n\geq 4$ we have
      $g=(2,2,1,\dotsc,\overset{n-2}{1})\in\Ji$,
      $m=(2,1,\dotsc,\overset{n-1}{1})\in\Mi$, and
      $g \Doppelpfeil m$.
\end{enumerate}
\end{lemma}
\begin{proof}
To prove the first statement, let $g=(2,1,\dotsc,\overset{n-1}{1})$ and
$m=(1,1,\dotsc,\overset{n}{1})$. Then $m^*=\apply{n}$ and
$g^*=\apply{n-1,1}$. From Lemma~\ref{First_Lemma}, we infer that
$m^*\in\Ji$, $g^*\in\Mi$ and $m^*\Doppelpfeil g^*$. Applying
Corollary~\ref{Corollary_3.6} to the latter expression, we obtain
$g\Doppelpfeil m$, as required; moreover, using partition conjugation,
it follows that $g\in\Ji$ and $m\in\Mi$.
Similarly, combining Corollary~\ref{Corollary_3.6} with
Lemma~\ref{Second_Lemma}, one can prove the second statement.
\end{proof}

\begin{lemma}\label{lem:new3.14}
Let $b\in\Nplus$, $d\in\N$; set $n=2b+d$.
Then we have $g=(2,\dotsc,\overset{b}{2},1,\dotsc,\overset{b+d}{1})\in\Ji$
of type~\ref{typeA} if $d=0$ and of type~\ref{typeB} if $d\geq1$.
Its unique lower cover is
$\tilde{g}=(2,\dotsc,\overset{b-1}{2},1,\dotsc,\overset{b+d}{1},1)$.
\end{lemma}
\begin{proof}
If $d=0$, then~$g$ is of type~\ref{typeA} and has got a unique cliff at
position~$b$, from which a brick may fall and produce
$\tilde{g}=(2,\dotsc,\overset{b-1}{2},1,1)$,
see transition rule~\eqref{item:transition-cliff}.
If $d\geq 1$, then~$g$ has got no cliffs but a unique slippery step at
position~$b$. Therefore its unique lower cover is obtained by applying
transition rule~\eqref{item:transition-step} to~$g$, that is, by
letting a brick slip across the step at position~$b$ to obtain
$\tilde{g}=(2,\dotsc,2,\overset{b}{1},\dotsc,\overset{b+d}{1},1)$.
\end{proof}

\begin{lemma}\label{lem:new3.15}
Let $b,d\in\N$ with $b\geq 2$; set $n\defeq 2b+d$.
Then we have $m=(b,1,\dotsc,\overset{b+d}{1},1)\in\Mi$ of type~\ref{typeIII}, having
$\tilde{m}=(b,2,1,\dotsc,\overset{b+d}{1})$ as its unique upper cover,
where the part $1,\dotsc,1$ only appears for $b+d\geq3$, that is,
$d\geq 1$, or $d=0$ and $b\geq 3$.
\end{lemma}
\begin{proof}
As $b\geq2$, we have $m\in\Mi$ of type~\ref{typeIII}. Since
partition conjugation is an order-antiautomorphism of~$\Lp$,
the unique upper cover of~$m$ can be obtained by conjugating the unique
lower cover of $m^* = (b+d+1,1,1,\dotsc,\overset{b}{1})\in\Ji$.
Since $b\geq 2$, the join-irreducible~$m^*$ has got a unique cliff in
its first position and no slippery step. Applying transition
rule~\eqref{item:transition-cliff}, we obtain its lower cover
$(b+d,2,1,\dotsc,\overset{b}{1})$, which we can conjugate to obtain
$\tilde{m}=(b,2,1,\dotsc,\overset{b+d}{1})$, the unique upper cover
of~$m$. Knowing that there \emph{is} a unique upper cover~$\tilde{m}$
of~$m$ from $m\in\Mi$, the fact that~$\tilde{m}$ looks as given before
can also be verified by checking that~$m$ is a lower cover
of~$\tilde{m}$. Indeed, by simply applying transition
rule~\eqref{item:transition-step} to the slippery step in the second
position of~$\tilde{m}$ when $b+d\geq 3$, or by applying transition
rule~\eqref{item:transition-cliff} to the cliff in the second position
of $\tilde{m}=(b,2)=(2,2)$ when $d=0$ and $b=2$, we see that we
get~$m$.
\end{proof}

\begin{lemma}\label{lem:edith-06-03-2023}
Let $b,d\in\N$, $b\geq 2$, and set $n\defeq 2b+d$.
Then, for
$g\defeq(2,\dotsc,\overset{b}{2},1,\dotsc,\overset{b+d}{1})$
and $m\defeq(b,1,\dotsc,\overset{b+d}{1},1)$ we have
$g\Doppelpfeil m$; moreover,
$m^*=(d+b+1,1,\dotsc,\overset{b}{1})\Doppelpfeil
 g^*=(d+b,b)$.
\end{lemma}
\begin{proof}
Clearly, the statement about~$m^*$ and~$g^*$ will follow from the one
about~$g$ and~$m$ by Corollary~\ref{Corollary_3.6}. We therefore let
$n=2b+d$ with $b\geq 2$, $d\geq 0$ and consider~$g$ and~$m$.
We have that $g\nleq m$ because the partial sums up to index~$b$
exhibit the relation $2b>2b-1=b+b-1$. From Lemma~\ref{lem:new3.14} we
know that $g\in\Ji$ has the unique lower cover
$\tilde{g} = (2,\dotsc,\overset{b-1}{2},1,\dotsc,\overset{b+d}{1},1)$;
from Lemma~\ref{lem:new3.15}  we infer that $m\in\Mi$ has the unique
upper cover $\tilde{m}=(b,2,1,\dotsc,\overset{b+d}{1})$.
\par
First we prove that $g\Hochpfeil m$. We shall verify that
$g\leq \tilde{m}$ by comparing the sequences of partial sums. For the
sum up to the index $j=1$ we have $2\leq b$ by hypothesis.
The sequence of partial sums of~$\tilde{m}$ continues after~$b$ with
the terms $\apply{\min(b+j,n)}_{j\geq 2}$. If $2\leq j\leq b$, then the
partial sum of~$g$ up to index~$j$ yields $2j=j+j\leq j+b$ since
$j\leq b$. For indices $j\geq b$ both partial sums coincide with the
value $j+b$, until they reach the common maximum~$n$. Hence
$\tilde{m}\geq g$, and thus we may rely on
Proposition~\ref{prop:arrows-in-doubly-founded-lattices}\eqref{item:m-ucover-above-g}
to show that $g\Hochpfeil m$.
\par
Our next goal is to demonstrate that $g\Runterpfeil m$. For this we
first show that $\tilde{g}\leq m$ by comparing the sequences of partial
sums. For~$m$ we obtain the sequence $(\min(n,b-1+j))_{j\geq 1}$, while
for~$\tilde{g}$ it begins with $(2j)_{1\leq j<b}$ and then continues
with $(\min(n,b-1+j))_{j\geq b}$.
For $1\leq j\leq b-1$ we have $2j=j+j\leq b-1+j$, and for indices
$j\geq b$ both sequences coincide. Thus we have established
$\tilde{g}\leq m$, and hence
Proposition~\ref{prop:arrows-in-doubly-founded-lattices}\eqref{item:g-lcover-below-m}
yields $g\Runterpfeil m$.
\end{proof}

\begin{corollary}\label{cor:all2-doublearrow-n-half-ones}
Let $d\in\N$, $b\in\Nplus$ and set $n\defeq 2b+d$.
Under these assumptions we have the two double arrow relations
$g = (2,\dotsc,\overset{b}{2},1,\dotsc,\overset{b+d}{1})\Doppelpfeil
 m = (b,1,\dotsc,\overset{b+d+1}{1})$
and $m^*=(b+d+1,1,\dotsc,\overset{b}{1})\Doppelpfeil g^*=(b+d,b)$.
\end{corollary}
\begin{proof}
If $b=1$, the first statement follows from
Lemma~\ref{Lemma3.7}\eqref{item:double-arrow-n-n-1-dual}, while
for $b\geq 2$ it follows from Lemma~\ref{lem:edith-06-03-2023}.
The second statement follows by dualisation via Corollary~\ref{Corollary_3.6}.
\end{proof}

\figurename~\ref{fig:std-cxt-lattice7} shows the arrow relations of
$\K\apply{\Lp[7]}$.
Using Lemmata~\ref{First_Lemma}, \ref{Second_Lemma} and~\ref{Lemma3.7},
one can verify some of those:
\begin{align*}
\apply{7}&\Doppelpfeil \apply{6,1}&
& \text{\footnotesize by Lemma~\ref{First_Lemma}}&&
&\apply{6,1} &\Doppelpfeil \apply{5,2}&
& \text{\footnotesize by Lemma~\ref{Second_Lemma}}\\
\apply{2,1,1,1,1,1}&\Doppelpfeil \apply{1,1,1,1,1,1,1}&
& \text{\footnotesize by Lemma~\ref{Lemma3.7}}&&
& \apply{2,2,1,1,1}&\Doppelpfeil \apply{2,1,1,1,1,1}&
& \text{\footnotesize by Lemma~\ref{Lemma3.7}}
\end{align*}
All remaining double arrows of $\K\apply{\Lp[7]}$ will be collected in the following remark,
whose proof will illustrate several similar cases.
In subsequent sections we introduce parameters to prove general
statements and avoid such repetitions.

\begin{remark}\label{rem:last-double-arrows-in-KL7}
 We consider \figurename~\ref{fig:lattice7} and the context given in
 \figurename~\ref{fig:std-cxt-lattice7}.
 The first double arrow $(2,2,2,1)\Doppelpfeil(3,1,1,1,1)$ and its dual
 $(5,1,1)\Doppelpfeil (4,3)$ are confirmed by
 Lemma~\ref{lem:edith-06-03-2023} for $b=3$ and $d=1$.
 It remains to verify
 \begin{align*}
   &\text{a)}\quad
  \begin{aligned}[t]
   (3,1,1,1,1) &\Doppelpfeil (2,2,2,1)
   \\
   (4,3) &\Doppelpfeil (5,1,1)
  \end{aligned}&&
  \text{b)}\quad \begin{aligned}[t]
   (4,1,1,1) &\Doppelpfeil (3,3,1)
   \\
   (3,2,2) &\Doppelpfeil (4,1,1,1)
  \end{aligned}&&
  \text{c)}\quad
   (3,3,1) \Doppelpfeil (3,2,2)
 \end{align*}
\end{remark}
\begin{proof}
First, we verify~a).
Using the  definition or inspecting
\figurename~\ref{fig:std-cxt-lattice7},
we get $g\defeq\apply{3,1,1,1,1} \nleq m\defeq\apply{2,2,2,1}$.
Clearly, we have $g\in\Ji[7]$ of type~\ref{typeC}
and $m\in\Mi[7]$ of type~\ref{typeII}.
Applying transition rule~\eqref{item:transition-cliff}, the
unique lower cover of~$g$ is $\tilde{g}=(2,2,1,1,1)$, which lies below~$m$.
Hence
Proposition~\ref{prop:arrows-in-doubly-founded-lattices}\eqref{item:g-lcover-below-m}
entails $g\Runterpfeil m$.
To prove the relation $g\Hochpfeil m$, we note that
the unique upper cover of~$m$ is
$\tilde{m}=(3,2,1,1)$, which is the conjugate of the unique lower cover
of $m^* = (4,3)$, as  explained in the proof of
Lemma~\ref{lem:new3.15}.
Using the  definition or inspecting \figurename~\ref{fig:lattice7},
we observe $g\le\tilde{m}$. Then, as $g\nleq m$,
Proposition~\ref{prop:arrows-in-doubly-founded-lattices}\eqref{item:m-ucover-above-g}
entails $g\Hochpfeil m$.
We conclude that $g\Doppelpfeil m$.
Furthermore, since $m^* = (4,3) $ and $g^* = (5,1,1)$,
Corollary~\ref{Corollary_3.6} gives the second double arrow in~a),
namely $(4,3)\Doppelpfeil(5,1,1)$.
\par

Next, we verify~b).
We repeat the arguments of~a)
and get $g\defeq\apply{4,1,1,1} \nleq m\defeq\apply{3,3,1}$
with  $g\in\Ji[7]$ of type~\ref{typeC} and $m\in\Mi[7]$ of type~\ref{typeII}.
Applying the same transition rules yields the unique lower cover
$\tilde{g}=(3,2,1,1)$ and the unique upper cover $\tilde{m}=(4,2,1)$.
Also, the definition or inspecting \figurename~\ref{fig:lattice7} gives
$g\le \tilde{m}$ and $\tilde{g}\le m$.
Therefore, since $g\nleq m$,
Proposition~\ref{prop:arrows-in-doubly-founded-lattices}
confirms $g\Doppelpfeil m$.
Furthermore, since $m^* = (3,2,2) $ and $g^* = (4,1,1,1)$,
Corollary~\ref{Corollary_3.6} gives the second double arrow in~b),
namely $(3,2,2)\Doppelpfeil(4,1,1,1)$.
\par

Finally,
to prove $g\defeq(3,3,1) \Doppelpfeil m\defeq(3,2,2)$,
observe that $g\in\Ji[7]$ of type~\ref{typeD}
and $m\in\Mi[7]$ of type~\ref{typeIV}.
Since $g_1+g_2 =6> m_1+m_2$, we get $g\nleq m$.
Furthermore, the unique lower cover $\tilde{g}=(3,2,2)$ of~$g$ is
obtained by applying transition rule~\eqref{item:transition-cliff}.
This gives both $\tilde{g} = m$ and $g=\tilde{m}$
where $\tilde{m}$ is the unique upper cover of $m$.
We conclude that $g\nleq m$, $\tilde{g}\le m$ and $g\le\tilde{m}$,
wherefore Proposition~\ref{prop:arrows-in-doubly-founded-lattices}
entails $g\Doppelpfeil m$.
\end{proof}

\subsection[Double arrows involving completely join-irreducibles of type A or B]{Double arrows involving completely join-irreducibles of type~\ref{typeA} or~\ref{typeB}}
The subsequent theorem exhibits the following one-to-one relation:
For every $g\in\Ji$ of type~\ref{typeA} or~\ref{typeB} there is exactly
one $m\in\Mi$ such that $g\Doppelpfeil m$.
This unique~$m$ has the form $m=(a,1,\dotsc,1)$
and for each~$m$ of this shape there is exactly one $p\in\Ji$ of
type~\ref{typeA} or~\ref{typeB}
such that $p\Doppelpfeil m$, namely $p=g$.

\begin{theorem}\label{thm:double-typeAB}
Let $b\in\N$, $k,\ell,n\in\Nplus$ be such that $b+\ell<n=b(k+1)+\ell k$ and
set $a=n-b-\ell$.
Then $g=(k+1,\dotsc,\overset{b}{k+1},k,\dotsc,\overset{b+\ell}{k})\in\Ji$
is of type~\ref{typeA} for $b=0$ and of type~\ref{typeB} for $b\ge1$,
$m=(a,1,\dotsc,\overset{b+\ell}{1},1)\in\Mi$,
and
$g\Doppelpfeil m$.
\par
Moreover, the previously described~$m$ are the
only $m\in\Mi$ satisfying $g\Doppelpfeil m$ for given $g\in\Ji$ of
type~\ref{typeA} or~\ref{typeB}.
\end{theorem}
\begin{proof}
First, we consider the cases $b+\ell=1$ and $k=1$ individually.
If $b+\ell =1$, then we get $\ell=1$, $b=0$, hence $1=b+\ell<n=k$ (since
$b=0$), and thus $g = (k) = (n)$ is of type~\ref{typeA} and $m=(n-1,1)$.
Hence the case $b+\ell=1$  was proved in Lemma~\ref{First_Lemma}.
Moreover, if $k=1$, then $0<n-b-\ell=b$, i.e., $b\geq 1$,
and~$g$ is of type~\ref{typeB}, wherefore this case was proved in
Corollary~\ref{cor:all2-doublearrow-n-half-ones}.
\par
Accordingly, it remains to show the statement for $k\ge2$ and
$b+\ell\ge2$. Then, as $k\geq 2$, we have $g\in\Ji$ because it is
completely join-irreducible of type~\ref{typeA} (for $b=0$) and of
type~\ref{typeB} (for $b\ge1$).
Moreover, $m\in\Mi$ is of type~\ref{typeIII} as
$a = n-b-\ell=b(k+1)+\ell k -b-\ell \ge 3b+2\ell-b-\ell = b+b+\ell\ge2$
due to $k\geq 2$ and $b+\ell\geq 2$.
Since~$g$ is `shorter' than~$m$, Lemma~\ref{lem:shorter-implies-nleq}
asserts $g\nleq m$.
\par
The unique lower cover of~$g$ is
$\tilde{g}=(k+1,\dotsc,\overset{b}{k+1},k,\dotsc,k,\overset{b+\ell}{k-1},1)$,
where the middle part $k,\dotsc,k$ only appears for $\ell\geq 2$.
The partition~$\tilde{g}$ is obtained by letting a brick fall from the
cliff in the last position of~$g$ (transition rule~\eqref{item:transition-cliff}).
The unique upper cover of~$m$ is
$\tilde{m}=(a,2,1,\dotsc,\overset{b+\ell}{1})$, where, to obtain~$m$,
we can let a brick fall from the cliff in the second position if
$b+\ell=2$, or let a brick slip across the slippery step if
$b+\ell\geq 3$ and the part $1,\dotsc,1$ is present.
We shall prove that $\tilde{g}\leq m$ and $g\leq\tilde{m}$;
then Proposition~\ref{prop:arrows-in-doubly-founded-lattices} will
entail $g\Doppelpfeil m$. By
Lemma~\ref{lem:top-typeIII-least-length}, $m$ is the largest
partition of length $b+\ell+1=\len(\tilde{g})$; hence $\tilde{g}\leq m$.
By Corollary~\ref{cor:subtop-typeIII-least-length}, $\tilde{m}$ is
the second largest partition of length $b+\ell=\len(g)$, while,
according to Lemma~\ref{lem:top-typeIII-least-length},
$p\defeq (n-(b+\ell-1),1,\dotsc,\overset{b+\ell}{1})$ is the largest.
Since $k\geq 2>1$ we have $g\neq p$; hence $g\leq \tilde{m}$.
\par
That there are no other $m\in\Mi$ in relation $g\Doppelpfeil m$ with~$g$
of type~\ref{typeA} or~\ref{typeB} will be proved in the subsequent
lemmata, the proof being concluded with
Corollary~\ref{cor:typeAB-doublearrow-implies-described}.
\end{proof}

\begin{lemma}\label{lem:typeAB-downarrow-implies-length-m=b+l+1-ending-in-1}
Let $b\in\N$, $k,\ell,n\in\Nplus$ be such that $b+\ell<n=b(k+1)+\ell k$
and set
$g\defeq(k+1,\dotsc,\overset{b}{k+1},k,\dotsc,\overset{b+\ell}{k})$.
If $m\in\Mi$ satisfies $g\Runterpfeil m$, then $\len(m) = b+\ell+1$
and the final entry of~$m$ is $m_{b+\ell+1}=1$.
\end{lemma}
\begin{proof}
Under the given conditions, $g$ is completely join-irreducible of
type~\ref{typeA} or~\ref{typeB},
see Corollary~\ref{cor:least-bounded-width-typeAB}.
Let us assume that $g\Runterpfeil m$ with $m\in\Mi$.
If $\len(m)\leq b+\ell$, then $m\in\Bw{b+\ell}$, and thus
Corollary~\ref{cor:least-bounded-width-typeAB} shows that $g\leq m$,
in contradiction to $g\Runterpfeil m$.
Therefore, $\len(m)\geq b+\ell+1$.
The unique lower cover~$\tilde{g}$ of~$g$ satisfies
$\len(\tilde{g})=b+\ell+1$.
By Proposition~\ref{prop:arrows-in-doubly-founded-lattices}\eqref{item:g-lcover-below-m},
our assumption $g\Runterpfeil m$ entails $\tilde{g}\leq m$, wherefore
$\len(\tilde{g})<\len(m)$ is impossible by
Lemma~\ref{lem:shorter-implies-nleq}. Consequently, we conclude
$\len(m)\leq\len(\tilde{g})=b+\ell+1$ and hence
$\len(m)=\len(\tilde{g})=b+\ell+1$,
which also yields $m_{b+\ell+1} \geq 1$.
Moreover, we have
$n-1=\sum_{i=1}^{b+\ell} \tilde{g}_i\leq
\sum_{i=1}^{b+\ell}m_i=n-m_{b+\ell+1}$
due to $\tilde{g}\leq m$ and $\tilde{g}_{b+\ell+1}=1$;
thus $m_{b+\ell+1}\leq 1$, and therefore $m_{b+\ell+1}=1$.
\end{proof}

\begin{lemma}\label{lem:typeAB-downarrow-typeI-III-implies-described}
Let $b\in\N$, $k,\ell,n\in\Nplus$ be such that $b+\ell<n=b(k+1)+\ell k$
and set
$g\defeq(k+1,\dotsc,\overset{b}{k+1},k,\dotsc,\overset{b+\ell}{k})$.
If $m\in\Mi$ is of type~\ref{typeI}, \ref{typeII} or~\ref{typeIII} and
satisfies $g\Runterpfeil m$, then~$m$ is of the form as described in
Theorem~\ref{thm:double-typeAB}.
\end{lemma}
\begin{proof}
We assume that the given~$g$ and $m\in\Mi$ satisfy $g\Runterpfeil m$.
From Lemma~\ref{lem:typeAB-downarrow-implies-length-m=b+l+1-ending-in-1}
we know that $\len(m) = b+\ell+1$ and that $m_{b+\ell+1}=1$.
If~$m$ is of type~\ref{typeI}, \ref{typeIII}, or~\ref{typeII} with
$\len(m)=2$, then~$m$ is of the form
$m=(a,1,\dotsc,\overset{b+\ell+1}{1})$
with $a\geq 1$. Since its total sum amounts to $n=a+(b+\ell)$, we have
$a=n-b-\ell$ and~$m$ is exactly of
the form described in Theorem~\ref{thm:double-typeAB}.
\par
Finally, let us consider the case that~$m$ is of type~\ref{typeII} with $\len(m)\geq 3$
and~$m$ has the form $m=(\kappa,\dotsc,\overset{b+\ell}{\kappa},1)$ where $b+\ell\geq 2$ and
$\kappa\geq2$. Thus, if $k=1$, then $g\leq m$ by
Corollary~\ref{cor:below-implies-dominance}, which contradicts $g\Runterpfeil m$.
Therefore, we know that $k\geq 2$.
Hence the unique lower cover~$\tilde{g}$ of~$g$ is obtained by
letting a brick fall from the cliff in position $b+\ell\geq2$,
and thus $\tilde{g} = (k+1,\dotsc,\overset{b}{k+1},k,\dots,k,\overset{b+\ell}{k-1},1)$
with its first entry satisfying $\tilde{g}_1 \ge k$.
Now $g\Runterpfeil m$ and
Proposition~\ref{prop:arrows-in-doubly-founded-lattices}\eqref{item:g-lcover-below-m}
entail $\tilde{g}\leq m$, hence $k \le\tilde{g}_1\leq m_1=\kappa$.
Using this, we can finally estimate
$\sum_{i=1}^{b+\ell-1} m_i=n-\kappa-1\le n-k-1
 < n-k =\sum_{i=1}^{b+\ell-1}\tilde{g}_i$,
which implies the contradiction $\tilde{g}\nleq m$.
Therefore, the third case, where~$m$ is of type~\ref{typeII} with
$\len(m)\geq 3$, is impossible and the lemma has been shown.
\end{proof}

\begin{corollary}\label{cor:BarrowIIimpossible}
Partitions $g\in\Ji$ of type~\ref{typeB} are not arrow-related to any
$m\in\Mi$ of type~\ref{typeII}.
\end{corollary}
\begin{proof}
   First, we prove that~$g$ of type~\ref{typeB},
   $m$ of type~\ref{typeII} and $g\DownArrow m$ is impossible.
   Note that any arrow relation for
   $g=(k+1,\dotsc,\overset{b}{k+1},k,\dotsc,\overset{b+\ell}{k})$ of
   type~\ref{typeB} requires $\len(m) > \len(g)=b+\ell$, because
   otherwise Corollary~\ref{cor:least-bounded-width-typeAB} yields
   $m\in\Bw{b+\ell}$ and $g \le m$,
   which by Proposition~\ref{prop:arrows-in-doubly-founded-lattices}
   prevents any arrow.
   Therefore, we have $\len(g)\ge2$ and thus $\len(m)\ge 3$.
   The proof of Lemma~\ref{lem:typeAB-downarrow-typeI-III-implies-described}
   confirms that a relation $g \Runterpfeil m$ with~$m$ of
   type~\ref{typeII} with $\len(m)\ge3$ is impossible.
   \par

   Finally, if~$g$ of type~\ref{typeB} and~$m$ of type~\ref{typeII}
   were such that $g\Hochpfeil m$,
   then the conjugates~$m^*$ of type~\ref{typeB} and $g^*$ of
   type~\ref{typeII} would satisfy $m^* \Runterpfeil g^*$ by
   Lemma~\ref{Lemma_3.5}, but this is impossible as shown before.
\end{proof}

The following simple observation will be used more than once.
\begin{lemma}\label{lem:upper-cover-of-type-III-IV}
Let $a,t,r,c\in\N$ be such that $a>t>r\geq 0$ and $c\geq2$ and set
$n\defeq a+ct +r$. In~$\Lp$ we have the covering relation
$m\defeq(a,t,\dotsc,\overset{c+1}{t},r)\prec
\tilde{m}\defeq(a,t+1,t,\dotsc,t,\overset{c+1}{t-1},r)$.
Furthermore, $\tilde{m}$ is the unique upper cover of $m$.
\end{lemma}
\begin{proof}
If $c=2$, we apply transition rule~\eqref{item:transition-cliff} to the
cliff in the second position of $\tilde{m}=(a,t+1,t-1,r)$, yielding
${m=(a,t,t,r)}$. If $c\geq 3$, we apply
rule~\eqref{item:transition-step} to the slippery step in the second
position of ${\tilde{m}=(a,t+1,t,\dotsc,t,t-1,r)}$.
In both cases the result is~$m$, and thus
Lemma~\ref{lem:description-of-lower-covers} yields $m\prec \tilde{m}$.
Moreover, $m$ is completely meet-irreducible of type~\ref{typeIII}
(if $t=1$) and~\ref{typeIV} (if $t\ge2$), wherefore $\tilde{m}$ is the
unique upper cover.
\end{proof}

\begin{lemma}\label{lem:typeAB-uparrow-typeIV-impossible}
Let $b\in\N$, $k,\ell,n\in\Nplus$ be such that $b+\ell<n=b(k+1)+\ell k$
and set
$g\defeq(k+1,\dotsc,\overset{b}{k+1},k,\dotsc,\overset{b+\ell}{k})$.
If $m\in\Mi$ is of type~\ref{typeIV} and satisfies $g\Runterpfeil m$,
then $g\Hochpfeil m$ fails.
\end{lemma}
\begin{proof}
We assume that $m=(a,t,\dotsc,t,r)$ of type~\ref{typeIV} satisfies
$g\Runterpfeil m$.
Then Lemma~\ref{lem:typeAB-downarrow-implies-length-m=b+l+1-ending-in-1}
implies that $\len(m)=b+\ell+1$ and that the last non-zero entry
of~$m$ equals~$1$. Hence if $r=0$, then $t=1$, which is forbidden for~$m$ of
type~\ref{typeIV}; we therefore know $r\geq 1$ and actually $r=1$ by
Lemma~\ref{lem:typeAB-downarrow-implies-length-m=b+l+1-ending-in-1}.
We thus have $m=(a,t,\dotsc,\overset{b+\ell}{t},1)$
with $b+\ell\geq 3$ and $a>t\geq 2$.
By Lemma~\ref{lem:upper-cover-of-type-III-IV}, the unique upper cover
of~$m$ is $\tilde{m}=(a,t+1,t,\dotsc,t,\overset{b+\ell}{t-1},1)$,
where the part $t,\dotsc,t$ appears as of $b+\ell\geq 4$.
Since $\len(g)=b+\ell < b+\ell+1=\len(\tilde{m})$, we have
$g\nleq\tilde{m}$ by Lemma~\ref{lem:shorter-implies-nleq}.
According to
Proposition~\ref{prop:arrows-in-doubly-founded-lattices}\eqref{item:m-ucover-above-g},
this renders $g\Hochpfeil m$ impossible.
\end{proof}

The following result shows that
Theorem~\ref{thm:double-typeAB}
indeed describes all double arrows $g\Doppelpfeil m$
involving partitions~$g$ of type~\ref{typeA} or~\ref{typeB}.
\begin{corollary}\label{cor:typeAB-doublearrow-implies-described}
Let $b\in\N$, $k,\ell,n\in\Nplus$ be such that $b+\ell<n=b(k+1)+\ell k$
and set
$g\defeq(k+1,\dotsc,\overset{b}{k+1},k,\dotsc,\overset{b+\ell}{k})$.
Every $m\in\Mi$ in relation $g\Doppelpfeil m$
is of the form described in
Theorem~\ref{thm:double-typeAB}.
\end{corollary}
\begin{proof}
Under the assumption $g\Doppelpfeil m$,
Lemma~\ref{lem:typeAB-uparrow-typeIV-impossible}
implies that~$m$ is not of type~\ref{typeIV}. Therefore, it is of
type~\ref{typeI}--\ref{typeIII}, and
Lemma~\ref{lem:typeAB-downarrow-typeI-III-implies-described}
establishes the desired claim.
\end{proof}

\subsection[Double arrows involving completely join-irreducibles of type C]{Double arrows involving completely join-irreducibles of type~\ref{typeC}}
The following result exhibits further one-to-one double arrows
$g\Doppelpfeil m$, originating from any~$g\in\Ji$ of type~\ref{typeC},
and covering two exceptional cases.
These are~$(n)$ and $(2,1,\ldots,1)$,
which are also discussed in Theorem~\ref{thm:double-typeAB}.
\begin{theorem}\label{thm:double-typeC}
Let $n,k,t,r\in\N$ be such that $2\leq k\leq n$ and
$n=t(k-1)+r$ where $0\leq r\leq k-2$. We have
$t\ge1$,
$g=(k,1,\dotsc,\overset{n-k+1}{1})\in\Ji$,
$m=(k-1,\dotsc,\overset{t}{k-1},r)\in\Mi$
and $g\Doppelpfeil m$.
\par
Moreover, if $g\in\Ji$ is of type~\ref{typeC}, $m\in\Mi$ and
$g\Doppelpfeil m$ holds, then~$m$ must be of the shape as described before.
\end{theorem}
\begin{proof}
The proof of $g\Doppelpfeil m$ relies on showing $m^*\Doppelpfeil g^*$
and then applying Corollary~\ref{Corollary_3.6}. The conjugate of~$m$
is $m^* = (t+1,\dotsc,\overset{r}{t+1},t,\dotsc,\overset{k-1}{t})$, the
one of~$g$ is
$g^* = (n-(k-1),1,\dotsc,\overset{k}{1})$.
We infer that $t\geq 1$, for otherwise
$n=t(k-1)+r=r\leq k-2\leq n-2$.
As $r\in\N$, $\ell=k-1-r\geq 1$ and $k-1\leq n-1<n$,
we can apply Theorem~\ref{thm:double-typeAB} to infer
$m^*\Doppelpfeil g^*$.
\par
To show, for $g\in\Ji$ of type~\ref{typeC} that there are no
other double arrows will require a series of little results, being
complete with Corollary~\ref{cor:double-typeC-described}.
\end{proof}

\begin{lemma}\label{lem:single-typeC-height-k-1}
Let $n,k\in\N$ be such that $2\leq k\leq n$ and set
$g\defeq(k,1,\dotsc,\overset{n-k+1}{1})$. If $m\in\Mi$ satisfies
$g\Runterpfeil m$ or $g\Hochpfeil m$, then $m_1= k-1$.
\end{lemma}
\begin{proof}
Let us assume that $m_1\geq k$. In this case
Corollary~\ref{cor:below-implies-dominance} directly implies $g\leq m$,
which, by Proposition~\ref{prop:arrows-in-doubly-founded-lattices},
makes $g\Runterpfeil m$ and $g\Hochpfeil m$ impossible.
\par
Let us now assume that $m_1\leq k-2$. We work with the first entries of
the unique lower cover~$\tilde{g}$ of~$g$ and the unique upper
cover~$\tilde{m}$ of~$m$.
We have $\hgt(m)=m_1\leq k-2<k-1=\tilde{g}_1$, hence $\tilde{g}\nleq m$.
Thus, by
Proposition~\ref{prop:arrows-in-doubly-founded-lattices}\eqref{item:g-lcover-below-m},
$g\Runterpfeil m$ is excluded. Likewise, we have
$\hgt(\tilde{m})\leq m_1+1\leq k-1<k=g_1$, hence $g\nleq\tilde{m}$.
Now
Proposition~\ref{prop:arrows-in-doubly-founded-lattices}\eqref{item:m-ucover-above-g}
excludes $g\Hochpfeil m$.
\end{proof}

\begin{corollary}\label{lem:single-typeC-I-II-described}
Let $n,k\in\N$ be such that $2\leq k\leq n$ and set
$g\defeq(k,1,\dotsc,\overset{n-k+1}{1})$. If $m\in\Mi$ is of
type~\ref{typeI} or~\ref{typeII} and satisfies
$g\Runterpfeil m$ or $g\Hochpfeil m$, then~$m$ is of the form as
described in Theorem~\ref{thm:double-typeC}.
\end{corollary}

\begin{lemma}\label{lem:CarrowIIIimpossible}
Partitions $g\in\Ji$ of type~\ref{typeC} are not arrow-related to any
$m\in\Mi$ of type~\ref{typeIII}.
\end{lemma}

\begin{proof}
		Let $g\in\Ji$ be of type~\ref{typeC},
		that is $g= (k,1,\dotsc,\overset{d+1}1)$ with $k\ge3$.
		Suppose that $m$ is of type~\ref{typeIII}, 
		with $g\Runterpfeil m$ or $g\Hochpfeil m$.
		Then $m = (a,1,\ldots,\overset{c+1}1)$ with $a,c\ge2$,
		and Lemma~\ref{lem:single-typeC-height-k-1}  implies $a=k-1$ and $c+1=d+2$.
		Then the unique upper cover of $m$ is $\tilde{m} = (k-1,2,1,\dotsc,\overset{c}1)$ by Lemma~\ref{lem:upper-cover-of-type-III-IV}.
		Hence  $\tilde{m}\ngeq g$, and Proposition~\ref{prop:arrows-in-doubly-founded-lattices}\eqref{item:m-ucover-above-g} proves $g \nHochpfeil m$.
		Therefore $g\Runterpfeil m$.
		Now, the unique lower cover of $g$ is $\tilde{g} = (k-1,2,1,\dotsc,\overset{d+1}1)$ since $k\ge3$.
		Then, we get $\len(\tilde{g}) = d+1 < c+1 = \len(m)$, and thus
		Lemma~\ref{lem:shorter-implies-nleq}  yields $\tilde{g}\nleq m$,
		contradicting $g\Runterpfeil m$ by Proposition~\ref{prop:arrows-in-doubly-founded-lattices}\eqref{item:g-lcover-below-m}.
\end{proof}

\begin{lemma}\label{lem:double-typeC-not-III-IV}
Let $n,k\in\N$ be such that $2\leq k\leq n$ and set
$g\defeq(k,1,\dotsc,\overset{n-k+1}{1})$. If $m\in\Mi$ is of
type~\ref{typeIII} or~\ref{typeIV}, then
$g\Doppelpfeil m$ is impossible.
\end{lemma}
\begin{proof}
	By Lemma~\ref{lem:CarrowIIIimpossible}, $g\Doppelpfeil m$
	with $m$ of type~\ref{typeIII} is impossible.
	Therefore, aiming for a contradiction, we assume that $g\Doppelpfeil m$
with $m$ of type~\ref{typeIV}. 
Then  $\len(m)\geq 3$, and, by
Lemma~\ref{lem:single-typeC-height-k-1}, we must have $m_1=k-1$,
i.e., $m=(k-1,t,\dotsc,\overset{\ell}{t},r)$ where
$0\leq r<t<k-1$, $t\ge2$ and $\ell\geq 3$.
Now, by Lemma~\ref{lem:upper-cover-of-type-III-IV}, the upper cover of~$m$
is $\tilde{m} = (k-1,t+1,t,\dotsc,t,t-1,r)$, implying that
$g_1=k>k-1=\tilde{m}_1$. Therefore, $g\nleq\tilde{m}$, and consequently
Proposition~\ref{prop:arrows-in-doubly-founded-lattices}\eqref{item:m-ucover-above-g}
excludes $g\Hochpfeil m$, in contradiction to~$g\Doppelpfeil m$.
\end{proof}

We can now finally observe that for~$g$ of type~\ref{typeC} all double
arrows $g\Doppelpfeil m$ were characterised in the first part of
Theorem~\ref{thm:double-typeC}.
\begin{corollary}\label{cor:double-typeC-described}
Let $n,k\in\N$ be such that $2\leq k\leq n$ and set
$g\defeq(k,1,\dotsc,\overset{n-k+1}{1})$. If $m\in\Mi$
satisfies $g\Doppelpfeil m$, then~$m$ is of the form as
described in Theorem~\ref{thm:double-typeC}.
\end{corollary}
\begin{proof}
Assuming $g\Doppelpfeil m$, Lemma~\ref{lem:double-typeC-not-III-IV}
yields that~$m$ cannot be of type~\ref{typeIV} nor of
type~\ref{typeIII}.
Therefore, $m$ is of type~\ref{typeI} or~\ref{typeII}, whence
Corollary~\ref{lem:single-typeC-I-II-described} implies that~$m$ is of
the form as given in Theorem~\ref{thm:double-typeC}.
\end{proof}

\subsection[Double arrows involving completely join-irreducibles of type D]{Double arrows involving completely join-irreducibles of type~\ref{typeD}}
\begin{theorem}\label{thm:double-typeCD}
Let $b\in\N$, $k,d,\ell\in\Nplus$ with $k\geq 3, b+\ell\ge2$ and
set $n\defeq b(k+1)+\ell k+d$.
Now, choose $2\leq t\leq \min(k-1,d+1)$, set
$a\defeq b+(b+\ell)(k-t)+t-1$ and
decompose $n-a=ct+r$ with $0\leq r<t$ and $c\in\N$.
Then we have
$g\defeq(k+1,\dotsc,\overset{b}{k+1},k,\dotsc,\overset{b+\ell}{k},1,\dotsc,\overset{b+\ell+d}{1})\in\Ji$
and $m\defeq (a,t,\dotsc,\overset{c+1}{t},r)\in\Mi$ of
type~\ref{typeIV} and $g\Doppelpfeil m$.
\par
Moreover, if $g\in\Ji$ is of type~\ref{typeD}, $m\in\Mi$ and
$g\Doppelpfeil m$ holds, then~$m$ must be of the shape as described before.
\end{theorem}
Note that if $m\in\Mi$ is of type~\ref{typeIV},
then its dual partition $m^*\in\Ji$ is of type~\ref{typeD}.
Since Theorem~\ref{thm:double-typeCD} exhibits double arrow
relationships between $\bigvee$-irreducibles of type~\ref{typeD} and
$\bigwedge$-irreducibles of type~\ref{typeIV}, the dual result of
Theorem~\ref{thm:double-typeCD} under partition conjugation is
contained within the original statement.
\par
Moreover, Theorem~\ref{thm:double-typeCD} represents a one-to-many
double arrow relationship, as can be observed from the example
$g\Doppelpfeil (4,3,3)$ and $g\Doppelpfeil(5,2,2,1)$ where
$g=(4,4,1,1)$ in both cases.
\begin{proof}
First, we check that $m\in\Mi$ is of type~\ref{typeIV}
by verifying the inequalities $c\geq b+\ell\geq 2$ and $a\geq k+b\geq t+1\geq 3$.
We have
$a-k=b+(b+\ell-1)(k-t)-1
      =(b-1+\ell)(k-t)+b-1
      \ge b +(b+\ell-2) \ge b$
since $k-t\geq 1$ and $b+\ell\ge 2$, hence
$a\geq k+b\geq k\geq t+1\geq 3$.
For the other inequality we see that
\[a+(b+\ell-1)t=b+(b+\ell)k-1=b(k+1)+\ell k-1=n-d-1
  \leq n-t=a+ct+r-t<a+ct,\]
where we have used that $d+1\geq t$ and $r-t<0$.
Thus, due to $t>0$, we have $b+\ell-1<c$, i.e., $b+\ell<c+1$ or
$2\leq b+\ell\leq c$. Hence $m$ is completely meet-irreducible of
type~\ref{typeIV}.
We also observe that~$g$ is completely join-irreducible
of type~\ref{typeD} since $b+\ell\geq 2$, $k\geq 3$ and $d,\ell\geq 1$.
\par
We have that $g\nleq m$ because summing up the first $b+\ell\leq c$
entries of~$g$ gives $b(k+1)+\ell k=n-d$, which is larger than the
corresponding value of~$m$, namely $a+(b+\ell-1)t=n-d-1$.
\par
The unique lower cover of~$g$ is
$\tilde{g}=(k+1,\dotsc,\overset{b}{k+1},k,\dotsc,k,\overset{b+\ell}{k-1},2,1,\dotsc,\overset{b+\ell+d}{1})$,
in which the part $k,\dotsc,k$ only appears for $\ell\geq 2$ and $1,\dotsc,1$ only appears for $d\geq 2$.
We obtain~$\tilde{g}$ from~$g$ by letting the brick in position~$b+\ell$
fall from the cliff (transition rule~\eqref{item:transition-cliff}).
Since $k-1\geq t\geq 2$, Lemma~\ref{lem:dominance-from-complement} with
position $j\defeq b+\ell$ and common sum
$s=b(k+1)+\ell k-1 = n-d-1=a+(b+\ell-1)t$ directly implies
$\tilde{g}\leq m$. Therefore,
Proposition~\ref{prop:arrows-in-doubly-founded-lattices}\eqref{item:g-lcover-below-m}
yields $g\Runterpfeil m$.
\par
By Lemma~\ref{lem:upper-cover-of-type-III-IV}, the unique upper cover
of~$m$ is $\tilde{m}=(a,t+1,t,\dotsc,t,\overset{c+1}{t-1},r)$.
Since $b+\ell\leq c$, $k\geq t+1$ and $1\leq t-1$,
applying Lemma~\ref{lem:dominance-from-complement} with
position $j\defeq b+\ell$ and the common sum
$s=b(k+1)+\ell k = a+(b+\ell-1)t+1$ shows $g\leq\tilde{m}$. Thus
Proposition~\ref{prop:arrows-in-doubly-founded-lattices}\eqref{item:m-ucover-above-g}
yields $g\Hochpfeil m$.
\par
That, for given $g\in\Ji$ of type~\ref{typeD}, there are no
other $m\in\Mi$ than the previous with $g\Doppelpfeil m$ will follow
from the following results, the proof being complete with
Corollary~\ref{cor:double-typeCD-described}.
\end{proof}

The following lemma in particular applies to completely
join-irreducible elements~$g$ of any type.
\begin{lemma}\label{lem:m-sum-b+l>=n-d-implies-g-leq-m}
Let $b,d\in\N$, $k,\ell\in\Nplus$ and set
$n\defeq b(k+1)+\ell k+d$. For
$g\defeq (k+1,\dotsc,\overset{b}{k+1},k,\dotsc,\overset{b+\ell}{k},1,\dotsc,\overset{b+\ell+d}{1})$
and $m\in\Mi$ we have that
$\hat{m}_{b+\ell}\defeq\sum_{i=1}^{b+\ell}m_i\geq n-d$ implies $g\leq m$.
\end{lemma}
\begin{proof}
As $k\geq 1$, we have $1\leq j\defeq b+\ell\leq
b(k+1)+\ell k=n-d\leq\hat{m}_j$.
There are $\kappa,r\in\N$ with $r<j$ such that~$\hat{m}_j$
decomposes modulo~$j$ as $\hat{m}_j = j\kappa+r$.
As $r<j$, we have
$jk=bk+\ell k\leq b(k+1)+\ell k=n-d\leq
\hat{m}_j=j\kappa+r<j(\kappa+1)$, hence $k < \kappa+1$, i.e.,
$k\leq \kappa$.
If $k=\kappa$, then $b+jk=b(k+1) + \ell k =
n-d\leq\hat{m}_j=r+j\kappa=r+jk$ and so
$b\leq r$; otherwise $\kappa\geq k+1$.
In both cases $\breve{g}\defeq
(k+1,\dotsc,\overset{b}{k+1},k,\dotsc,\overset{j}{k})$ lies pointwise below
$p\defeq(\kappa+1,\dotsc,\overset{r}{\kappa+1},\kappa,\dotsc,\overset{j}{\kappa})$.
Now we have $p,\breve{m}\defeq(m_1,\dotsc,m_j)\in \Bw[\hat{m}_j]{j}$, and since
$\hat{m}_j = \kappa j+r$ with $0\leq r<j$ and $1\leq j\leq\hat{m}_j$,
Lemma~\ref{lem:least-bounded-width} implies
$p\leq \breve{m}$.
Hence we obtain
$\sum_{i=1}^s g_i=\sum_{i=1}^s \breve{g}_i\leq \sum_{i=1}^s p_i
\leq \sum_{i=1}^s \breve{m}_i = \sum_{i=1}^s m_i$ for each~$s$ with $1\leq s\leq j$.
If $\len(m)>j$, then for each $j<i\leq\len(m)$ we have
$g_i\leq 1\leq m_i$, and thus
$\sum_{i=1}^sg_i =\sum_{i=1}^j g_i + \sum_{i=j+1}^s g_i
\leq\sum_{i=1}^j m_i + \sum_{i=j+1}^s m_i =\sum_{i=1}^s m_i$ for any
$s>j$. Therefore, we conclude $g\leq m$.
\end{proof}

Also the following lemma applies to any possible completely
join-irreducible partition~$g$.
\begin{lemma}\label{lem:g-downarrow-m-implies-m-sum-b+l=n-d-1}
Let $b,d\in\N$, $k,\ell\in\Nplus$ with $k\ge2$ and set $n\defeq b(k+1)+\ell k +d$.
If\/ $k=2$ and $d\geq1$, we additionally assume $b=0$.
Define $g\defeq (k+1,\dotsc,\overset{b}{k+1},k,\dotsc,\overset{b+\ell}{k},1,\dotsc,\overset{b+\ell+d}{1})$
and let $m\in\Mi$ be such that $g\Runterpfeil m$.
Then $\hat{m}_{b+\ell}\defeq\sum_{i=1}^{b+\ell}m_i= n-(d+1)<n$.
\end{lemma}
\begin{proof}
Let $j\defeq b+\ell$ and assume $g\Runterpfeil m$. By
Lemma~\ref{lem:m-sum-b+l>=n-d-implies-g-leq-m}, we must
have $\hat{m}_j\leq n-d-1$ because $\hat{m}_j\geq n-d$ would entail
$g\leq m$, in contradiction to $g\Runterpfeil m$, see
Proposition~\ref{prop:arrows-in-doubly-founded-lattices}\eqref{item:g-lcover-below-m}.
On the other hand, we know from
Proposition~\ref{prop:arrows-in-doubly-founded-lattices}\eqref{item:g-lcover-below-m},
that $\tilde{g}\leq m$ holds for the unique lower cover~$\tilde{g}$
of~$g$, wherefore, we have $\hat{m}_j\geq \sum_{i=1}^j \tilde{g}_i\eqdef N$.
It only remains to observe that the latter sum indeed amounts to
$n-(d+1)$.
We do this by considering three different cases: if
$d=0$, then $g=(k+1,\dotsc,\overset{b}{k+1}, k,\dotsc,\overset{j}{k})$,
$\tilde{g}=(k+1,\dotsc,\overset{b}{k+1},k,\dotsc ,k,\overset{j}{k-1},1)$ and $N=b(k+1) +\ell k -1=n-1=n-d-1$.
Next, if $d\geq1$ and $k=2$, then  $b=0$, $j=\ell$, $n=\ell k +d = 2\ell+d$ and
$g= (2, \dotsc,\overset{\ell}{2}, 1,\dotsc,\overset{\ell+d}{1})$,
$\tilde{g} = (2, \dotsc,\overset{j-1}{2}, 1,\dotsc,\overset{\ell+d}{1},1)$
and
$N = 2\ell-1=b(k+1) +\ell k-1=n-d-1$.
The last case is $d\geq 1$ and $k\geq 3$.
Now~$g$ is of the general form given in the statement
and its unique lover cover is
$\tilde{g} = (k+1,\dotsc,\overset{b}{k+1},k,\dotsc,k,\overset{b+\ell}{k-1},2,1,\dotsc,\overset{b+\ell+d}{1})$
where the part $k,\dotsc,k$ only appears for $\ell\ge2$ and $1,\dotsc,1$
only for $d\ge2$. Again, we get $N=b(k+1) + \ell k -1 = n-d-1$ as
$n=b(k+1) +\ell k +d$.
Therefore, we conclude that in all cases mentioned in the lemma
the inequalities $n>n-d-1\geq \hat{m}_j\geq N = n-d-1$, i.e.,
$n>n-d-1=\hat{m}_j$, hold.
\end{proof}

The subsequent lemma shows in particular that if $g\in\Ji$ is of
type~\ref{typeD}, $m\in\Mi$ is of type~\ref{typeI}, \ref{typeII}
or~\ref{typeIII}, then $g\Runterpfeil m$ fails.
\begin{lemma}\label{lem:single-typeC-b>=2-I-II}
Let $b\in\N$, $k,\ell,d\in\Nplus$ be such that $k\geq 3$, $b+\ell\geq 2$ and set
$n\defeq b(k+1) +\ell k+d$. Consider
$g\defeq(k+1,\dotsc,\overset{b}{k+1},k,\dotsc,\overset{b+\ell}{k},1,\dotsc,\overset{b+\ell+d}{1})$
and any $m\in\Mi$. If $g\Runterpfeil m$, then
there are $a,t,c,r\in\N$ such that
$n= a+ct+r$,
$m= (a,t,\dotsc,\overset{c+1}{t},r)$,
$a>t>r\geq 0$, $c\geq 2$,
$2\leq t\leq k-1$ and
$a=b+(b+\ell)(k-t) +t-1$.
In particular, $m$ is of type~\ref{typeIV}. 
Moreover, if\/~$m$ is such that $t>d+1$, then
$r=d+1\geq 2$, $b+\ell=c+1$ and $g\Hochpfeil m$ fails.
\end{lemma}
\begin{proof}
Suppose that $g\Runterpfeil m$; then
Proposition~\ref{prop:arrows-in-doubly-founded-lattices}\eqref{item:g-lcover-below-m}
entails $\tilde{g}\leq m$ for the unique lower cover~$\tilde{g}$
of~$g$. Since $b+\ell\geq 2$, we thus have $g_1=\tilde{g}_1\leq m_1\eqdef h$.
If~$m$ were of type~\ref{typeI} or~\ref{typeII}, then
$m=(h,\dotsc,h,r)$ for some $0\leq r<h$, and thus~$m$ would be the
largest member of $\Bh{h}$ (see
Lemma~\ref{lem:largest-bounded-height}).
Since $\hgt(g)=g_1\leq h$, we would have $g\in\Bh{h}$, hence $g\leq m$,
contradicting $g\Runterpfeil m$ by
Proposition~\ref{prop:arrows-in-doubly-founded-lattices}\eqref{item:g-lcover-below-m}.
\par
Therefore, $m$ is neither of type~\ref{typeI} nor~\ref{typeII}, thus is
must be of type~\ref{typeIII} or~\ref{typeIV}. Hence it has the form
$m=(a,t,\dotsc,\overset{c+1}{t},r)$ where $a>t>r\geq 0$ and
$c\geq 2$. The total sum of~$m$ being~$n$, we note that $n=a+ct+r$.
If $t\geq k+1$, then $a>t\geq k+1>k\geq 3>1$ and hence $g\leq m$ by
Corollary~\ref{cor:below-implies-dominance}. By
Proposition~\ref{prop:arrows-in-doubly-founded-lattices}\eqref{item:g-lcover-below-m},
this is in contradiction to $g\Runterpfeil m$, wherefore $t\leq k$.
As $k\geq 3$, Lemma~\ref{lem:g-downarrow-m-implies-m-sum-b+l=n-d-1}
shows that the partial sum of~$m$ up to position~$b+\ell$ is
$\hat{m}_{b+\ell}=n-(d+1)<n$. Since $\hat{m}_{b+\ell}<n$, we must have
$b+\ell\leq c+1$, hence
$\hat{m}_{b+\ell} = a+(b+\ell-1)t = n-(d+1) = a+ct +r -(d+1)$
and $\hat{m}_{b+\ell-1}=n-(d+1)-t$.
Since $g\Runterpfeil m$,
Proposition~\ref{prop:arrows-in-doubly-founded-lattices}\eqref{item:g-lcover-below-m}
yields $\tilde{g}\leq m$ for the unique lower cover~$\tilde{g}$ of~$g$.
Thus,
$n-(d+k) = n-\sum_{i=b+\ell}^{b+\ell+d}g_i =
n-\sum_{i=b+\ell}^{b+\ell+d}\tilde{g}_i =
\sum_{i=1}^{b+\ell-1}\tilde{g}_i \leq \hat{m}_{b+\ell-1}=n-(d+1)-t$,
i.e., $t\leq k-1$. Since $t>r\geq 0$, we have $t\geq 1$.
With the intent to derive a contradiction, we assume that $t=1$, hence
$m=(a,1,\dotsc,1)$.
We know that
$\tilde{g}=(k+1,\dotsc,\overset{b}{k+1},k,\dotsc,k,\overset{b+\ell}{k-1},2,1,\dotsc,\overset{b+\ell+d}{1})
          \leq m$, thus
$d\geq 1$ and Lemma~\ref{lem:shorter-implies-nleq} imply
$\len(g)= \len(\tilde{g})\geq \len(m)$; hence our assumption
$t=1$ and Lemma~\ref{lem:top-typeIII-least-length} yield $g\le m$,
contradicting $g\Runterpfeil m$ by
Proposition~\ref{prop:arrows-in-doubly-founded-lattices}\eqref{item:g-lcover-below-m}.
Therefore, $2\leq t\leq k-1$, and~$m$ has type~\ref{typeIV}.
Moreover, from $\hat{m}_{b+\ell}=a+(b+\ell-1)t=n-(d+1)$ we infer
\begin{align*}
a &= n-(d+1)-(b+\ell-1)t=(b(k+1)+\ell k+d)-d-1-(b+\ell)t+t\\
&=b+b(k-t) +\ell (k -t) +t-1
=b+(b+\ell)(k-t) +t-1.
\end{align*}
\par
From
$\hat{m}_{b+\ell} = a+(b+\ell-1)t = a+ct +r -(d+1)$
we infer that
$d+1 + (b+\ell-1)t = ct +r$ and hence $d+1\equiv r\pmod{t}$.
If we now additionally assume that $d+1<t$, then $d+1=r$, and we may
cancel this common summand from both sides of the equation, leading to
$(b+\ell-1)t=ct$, or $b+\ell-1=c$, i.e., $b+\ell=c+1$.
Consequently, we have $m=(a,t,\dotsc,\overset{b+\ell}{t},d+1)$, being
of type~\ref{typeIV} with $d+1\geq 2$, and thus the
unique upper cover of~$m$ is
$\tilde{m}=(a,t+1,t,\dotsc,t,\overset{b+\ell}{t-1},d+1)$, see
Lemma~\ref{lem:upper-cover-of-type-III-IV}.
Therefore, we obtain
$\sum_{i=1}^{b+\ell}\tilde{m}_i = n-(d+1)<n-d = b(k+1) +\ell k =
\sum_{i=1}^{b+\ell}g_i$ and hence $g\nleq \tilde{m}$.
By Proposition~\ref{prop:arrows-in-doubly-founded-lattices}\eqref{item:m-ucover-above-g}
$g\Hochpfeil m$ fails.
\end{proof}

The following corollary shows that all double arrows $g\Doppelpfeil m$
involving $g\in\Ji$ of type~\ref{typeD} were described in
Theorem~\ref{thm:double-typeCD}.
In particular, dualising the statements of
Theorem~\ref{thm:double-typeCD} yields double arrow relations that are
already described as a part of that theorem.
\begin{corollary}\label{cor:double-typeCD-described}
Let $b\in\N$, $k,d,\ell\in\Nplus$ be such that $k\geq 3$, $b+\ell\geq2$
and define $n\defeq b(k+1)+\ell k+d$, as well as
$g\defeq(k+1,\dotsc,\overset{b}{k+1},k,\dotsc,\overset{b+\ell}{k},1,\dotsc,\overset{b+\ell+d}{1})$.
If $m\in\Mi$ satisfies $g\Doppelpfeil m$, then~$m$ is of
type~\ref{typeIV} of the shape as shown in
Theorem~\ref{thm:double-typeCD}.
\end{corollary}
\begin{proof}
Let us assume that $g\Doppelpfeil m$. According to Lemma~\ref{lem:single-typeC-b>=2-I-II},
the partition $m\in\Lp$ is of the form
$m=(a,t,\dotsc,\overset{c+1}{t},r)$ with
$a>t>r\geq 0$, $c\geq 2$, $n-a=ct+r$, $2\leq t\leq k-1$, and
$a=b+(b+\ell)(k-t)+t-1$. 
Moreover, if $t>d+1$, then $g\Hochpfeil m$ would fail by
Lemma~\ref{lem:single-typeC-b>=2-I-II}, contradicting $g\Doppelpfeil m$.
Therefore, we have $2\leq t\leq \min(k-1,d+1)$,
whence~$m$ is of type~\ref{typeIV} with exactly the shape as given in
Theorem~\ref{thm:double-typeCD}.
\end{proof}

We summarise our understanding of the double arrow relation in
\tablename~\ref{tab:DoubleArrowCaracterizationSummary}.
\begin{table}[ht]
\centering
\begin{tabular}{llcl}
\toprule
Result & $g\in\Ji$ & Arrow  & $m\in\Mi$ \\
\midrule
Theorem~\ref{thm:double-typeAB}: &
 $(k,\dotsc,\overset{\ell}{k})$,\, $k\ge2$
 & $\Doppelpfeil$ &
$(n-\ell,1,\dotsc,\overset{\ell}{1},1)$
\\ \midrule
Theorem~\ref{thm:double-typeAB}: &
 $(k+1,\dotsc,\overset{b}{k+1},k,\dotsc,\overset{b+\ell}{k})$,\, $k\ge1$
 & $\Doppelpfeil$ &
$(n-b-\ell,1,\dotsc,\overset{b+\ell}{1},1)$
\\ \midrule
Theorem~\ref{thm:double-typeC}: &
 $(k,1,\dotsc,\overset{n-k+1}{1})$,\, $k\ge2$
 & $\Doppelpfeil$ &
$(k-1, \dotsc,\overset{c}{k-1},r)$
\\ \midrule
Theorem~\ref{thm:double-typeCD}: &
$(k,\dotsc,\overset{\ell}{k},1,\dotsc,\overset{\ell+d}{1})$,\, $k\ge3$, $\ell\ge2$
 & $\Doppelpfeil$ &
 $(a,t,\dotsc,\overset{c+1}{t},r)$,\, $c\ge2$
\\ \midrule
Theorem~\ref{thm:double-typeCD}: &
$(k+1,\dotsc,\overset{b}{k+1},k,\dotsc,\overset{b+\ell}{k},1,\dotsc,\overset{b+\ell+d}{1})$,\, $k\ge3$
 & $\Doppelpfeil$ &
 $(a,t,\dotsc,\overset{c+1}{t},r)$,\, $c\ge2$
\\
\bottomrule
\end{tabular}
\caption{A summary of the~$\Doppelpfeil$ characterisation results where
$b,\ell,d,c,a,t\in\Nplus$ with $a\ge k>t>r\geq0$ and $t\geq2$.
The precise value of~$a$ is stated in Theorem~\ref{thm:double-typeCD}.}
\label{tab:DoubleArrowCaracterizationSummary}
\end{table}

\section{Results regarding single arrows}
\label{sect:single-arrows}

Since results for up-arrows can be deduced from down-arrows by
partition conjugation and Lemma~\ref{Lemma_3.5}, we focus here on the
complete description of the relation $g\Runterpfeil m$ where $g\in\Ji$
and $m\in\Mi$. As, by definition, $g\Doppelpfeil m$ implies
$g\Runterpfeil m$, the results from Section~\ref{sect:double-arrows}
yield part of the description of the down-arrows. Subsequently, we need
to concentrate on finding the remaining `single' down-arrows, where
$g\Runterpfeil m$ but not $g\Hochpfeil m$. Our strategy will be to
derive necessary conditions that follow from $g\Runterpfeil m$, to
isolate the cases that were already covered in
Section~\ref{sect:double-arrows}, and then to prove in the remaining
cases from a subset of the necessary conditions that $g\Runterpfeil m$
actually holds. We organise our investigations according to the
different types $g\in\Ji$ may have.
\par

For $g\in\Ji$ of type~\ref{typeA} or~\ref{typeB},
Lemma~\ref{lem:typeAB-downarrow-typeI-III-implies-described}
confirms that all down-arrows involving~$m$ of types~\ref{typeI},
\ref{typeII} and~\ref{typeIII} are contained in
Theorem~\ref{thm:double-typeAB}
and hence are double arrows.
Therefore, we here only discuss down-arrows between~$g$ of
type~\ref{typeA} or~\ref{typeB} and~$m\in\Mi$ of type~\ref{typeIV}.
In this respect, Lemma~\ref{lem:NecCondGABdownM4} gives necessary conditions for
the down-arrows;
these necessary conditions will be shown to be sufficient in
Theorem~\ref{thm:TypeABDownIV}.

\begin{lemma}\label{lem:NecCondGABdownM4}
  Let $b\in\N$, $n,k,\ell\in\Nplus$ be such that $b+\ell<n=b(k+1)+\ell k$.
  The partition
  $g= (k+1,\dotsc,\overset{b}{k+1},k,\dotsc, \overset{b+\ell}{k})$
  is of type~\ref{typeA} or~\ref{typeB}.
  If $m\in\Mi$ is of type~\ref{typeIV} and satisfies $g\Runterpfeil m$,
  then $k,b+\ell\ge3$ and $m=(a,t,\dotsc,\overset{b+\ell}{t},1)$ with $a>t$ and
  $2\le t \le k-1$.
\end{lemma}
\begin{proof}
  Lemma~\ref{lem:typeAB-downarrow-implies-length-m=b+l+1-ending-in-1}
  yields
  $\len(m)=b+\ell+1$ and that $m=(a,t,\dotsc,t,r)$
  with $a>t>r\geq 0$ and $t\geq 2$ must end with~$1$. If $r=0$, then
  this implies the contradiction $t=1$, hence $r\geq 1$ and thus $r=1$
  by Lemma~\ref{lem:typeAB-downarrow-implies-length-m=b+l+1-ending-in-1}.
  Therefore, $m=(a,t,\dotsc,\overset{b+\ell}{t},1)$ with $a>t\geq 2$,
  and since two values~$t$ must appear, we obtain $b+\ell\ge3$.
  Finally, we prove $t\le k-1$ in two steps.
  First, we cannot have $t\ge k+1$ because under this assumption the
  partial sum of~$m$ exceeds $(k+1)(b+\ell)+1>n$ as $a>t\ge k+1$.
  Thus, we have $2\leq t\leq k$, i.e., $k\geq 2$,
  wherefore~$g$ has a cliff in position $b+\ell$,
  and by transition rule~\eqref{item:transition-cliff} the
  unique lower cover of~$g$ is
  $\tilde{g}=(k+1,\dotsc,\overset{b}{k+1},k,\dotsc,k,\overset{b+\ell}{k-1},1)$.
  Our assumption $g\Runterpfeil m$ and
  Proposition~\ref{prop:arrows-in-doubly-founded-lattices}\eqref{item:g-lcover-below-m}
  imply $\tilde{g}\leq m$, hence we get
  $n-k = \sum_{i=1}^{b+\ell-1} \tilde{g}_i\leq \sum_{i=1}^{b+\ell-1}
  m_i =n-t-1$, i.e., $t\leq k-1$. Therefore, $k\geq t+1\geq 3$, finishing the proof.
\end{proof}

We now show that the conditions in Lemma~\ref{lem:NecCondGABdownM4} are
sufficient for a down-arrow between~$g$ of type~\ref{typeA}
or~\ref{typeB} and~$m$ of type~\ref{typeIV}, and that this arrow is not
a double arrow.

\begin{theorem}\label{thm:TypeABDownIV}
  Let $b\in\N$, $k,t,\ell\in\Nplus$ with $k,b+\ell\ge3$
  and $2\le t\leq k-1$; set $n=b(k+1)+\ell k$ and $a = n-1 - (b+\ell-1)t$.
  Then the partitions
  $g=(k+1,\dotsc,\overset{b}{k+1}, k,\dotsc,\overset{b+\ell}{k})$
  and
  $m=(a,t,\dotsc,\overset{b+\ell}{t},1)$ satisfy
  $g\in\Ji$ being of type~\ref{typeA} or~\ref{typeB},
  $m\in\Mi$ being of type~\ref{typeIV} with $a\geq t+2$, and
  $g\DownArrow m$ but not $g\Hochpfeil m$.
\end{theorem}
\begin{proof}
As $k\ge3$ we have that $g\in\Ji$ is of type~\ref{typeA} (if $b=0$)
or~\ref{typeB} (if $b\ge1$).
In Lemma~\ref{lem:NecCondGABdownM4} we saw that the unique lower cover of~$g$ is
$\tilde{g} =
(k+1,\dotsc,\overset{b}{k+1},k,\dotsc,k,\overset{b+\ell}{k-1},1)$.
Since~$g$ is shorter than~$m$,
Lemma~\ref{lem:shorter-implies-nleq} yields $g\nleq m$.
Next, we verify that  $m\in\Mi$ is of type~\ref{typeIV}.
First, given $b+\ell\ge3$, we get that~$t$ appears $(b+\ell-1)\ge2$ times.
Next, we have $m\in\Lp$ since
$a = n-1-(b+\ell-1)t$.
Furthermore, from
$n=(b+\ell)(k-1) +2b+\ell$, $t\le k-1$ and $b+\ell\ge3$,
we get
$a-t = n - (b+\ell)t -1 \geq n -(b+\ell)(k-1) -1
     = 2b+\ell-1 =(b+\ell) +b-1 \geq 2$,
hence $a\geq t+2>t$ and~$m$ is of type~\ref{typeIV}.
\par
We now prove $\tilde{g}\le m$.
Note that the partial sums of both~$\tilde{g}$ and~$m$ reach $n-1$ at position $b+\ell$.
Therefore, since $\tilde{g}_i\ge k-1\ge t= m_i$ for $2\leq i\leq b+\ell$, 
Lemma~\ref{lem:dominance-from-complement} with $j=b+\ell$ gives $\tilde{g}\le m$.
Hence Proposition~\ref{prop:arrows-in-doubly-founded-lattices}\eqref{item:g-lcover-below-m}
yields $g \Runterpfeil m$.
Finally, as $n\geq k(b+\ell)\geq 3(b+\ell)>b+\ell$,
Lemma~\ref{lem:typeAB-uparrow-typeIV-impossible} shows that
$g\Hochpfeil m$ in Theorem~\ref{thm:TypeABDownIV} is impossible.
\end{proof}

Lemmata~\ref{lem:typeAB-downarrow-typeI-III-implies-described},
\ref{lem:NecCondGABdownM4} and Theorem~\ref{thm:TypeABDownIV} characterise all
$m\in\Mi$ in relation $g\DownArrow m$ with~$g$ of type~\ref{typeA}
or~\ref{typeB}.
\par

We are now starting to deal with $g\Runterpfeil m$ for
partitions~$g\in\Ji$ of type~\ref{typeC}.
Theorem~\ref{thm:double-typeC} and
Corollary~\ref{lem:single-typeC-I-II-described} cover the case
when $m\in\Mi$ is of type~\ref{typeI} or~\ref{typeII}.
Here the down-arrow automatically is a double arrow.
We therefore focus in the subsequent two results
on~$m$ of types~\ref{typeIII} or~\ref{typeIV}. As done
earlier we first derive necessary conditions from $g\Runterpfeil m$,
which we then prove to be sufficient for a down-arrow that is not a
double arrow.
\begin{lemma}\label{lem:NecCondCone}
  Let $k,d\in\Nplus$ with $k\ge3$, set $n=k+d$ and
  let $g=(k,1,\dotsc,\overset{1+d}1)\in\Ji$.
  Suppose $g\DownArrow m$
  with  $m\in\Mi$ of type~\ref{typeIII} or~\ref{typeIV},
  then $k\ge 4$,
  $m=(k-1,t,\dotsc,\overset{c+1}t,r)$ for $r,t,c\in\N$
  with $0\le r<t\le k-2$, $t\ge2$, $c\ge2$,
  i.e., $m$ is of type~\ref{typeIV}, and $d+1\ge 2t$.
\end{lemma}
\begin{proof}
  By Lemma~\ref{lem:single-typeC-height-k-1}, $m$ starts with $m_1=k-1=n-(d+1)$.
  Since $m$ is of type~\ref{typeIII} or~\ref{typeIV}, we have
  $m=(m_1,t,\dotsc,\overset{c+1}{t},r)$ where
  $k-1=m_1>t>r\geq 0$ and $c\geq 2$.
  From $g\Runterpfeil m$ we infer by
  Proposition~\ref{prop:arrows-in-doubly-founded-lattices}\eqref{item:g-lcover-below-m}
  that $\tilde{g}\leq m$, where $\tilde{g}$ denotes the unique lower
  cover of~$g$. Hence we obtain
  $k+1=(k-1)+2=\tilde{g}_1+\tilde{g}_2\leq m_1+m_2 = k-1 +t$, i.e.,
  $t\geq 2$ and~$m$ is of type~\ref{typeIV}.
  As $m_1 = n-(d+1)$, the total sum of~$m$ equals
  $n=n-(d+1)+ct+r$, thus $d+1=ct+r\geq ct\geq 2t$ since $r\geq 0$ and
  $c\geq 2$.
  Finally, we get $k\geq 4$ from $2\leq t\leq k-2$.
\end{proof}

\begin{theorem}\label{thm:TypeCdown1}
  Let $k,d,t\in \Nplus$ with $k\ge4$,
  $2\leq t\leq k-2$, $2t\le d+1$ and set
  $n=k+d$.
  Let us decompose $n-(k-1)=ct + r$ with $c,r\in\N$ such that
  $0\leq r<t$.
  Then $c\geq 2$, and for
  $g= (k,1,\dotsc,\overset{d+1}{1})$ and
  $m = (k-1,t,\dotsc,\overset{c+1}{t},r)$ we have
  $g\in\Ji$ being of type~\ref{typeC},
  $m\in\Mi$ being of type~\ref{typeIV}, and
  $g \Runterpfeil m$, but not $g\Hochpfeil m$.
\end{theorem}
Observe that $g=(5,1,1,1,1,1)$ satisfies
$g\DownArrow (4,3,3)$ and $g\DownArrow(4,2,2,2)$.
Hence Theorem~\ref{thm:TypeCdown1} represents a one-to-many
relationship.
The number of distinct $m\in\Mi$ such that $g\DownArrow m$ with a
fixed~$g$ grows with~$n$, e.g., if $n=16$ and $g=(7,1,\dotsc,1)$, then
$g\DownArrow m$ for
$m\in\set{(6,5,5), (6,4,4,2), (6,3,3,3,1), (6,2,2,2,2,2)}$.
\begin{proof}
  First, we check that $m\in\Mi$. By the definition of~$c$ and~$r$, we
  have $m\in\Lp$. Moreover, since $t>r$, $n=k+d$ and $d+1\geq 2t$,
  we obtain
  $(c-1)t>ct+r-2t=n-(k-1)-2t=d+1-2t\geq 0$, which yields $c-1>0$, i.e.,
  $c\geq 2$.
  Thus, $m$ is completely meet-irreducible of type~\ref{typeIV}.
  We have  $g\nleq  m$ as the first entries are  $g_1 = k>k-1=m_1$.
  Since $k\geq4$ and $d\geq1$, we observe that~$g$ is completely
  join-irreducible of type~\ref{typeC} and that its unique lower cover is
  $\tilde{g} = (k-1,2,1,\dotsc,\overset{d+1}{1})$,
  which is obtained from~$g$ by letting the brick in position~$1$ fall
  from the cliff (transition rule~\eqref{item:transition-cliff}).
  Since $m_1=k-1=\tilde{g}_1$ and $m_i=t\geq 2>1>0$ for
  $2\leq i\leq c+1$, we obtain $\tilde{g}\leq m$ from
  Corollary~\ref{cor:below-implies-dominance}.
  Thus,
  Proposition~\ref{prop:arrows-in-doubly-founded-lattices}\eqref{item:g-lcover-below-m}
  confirms $g\Runterpfeil m$.
  \par

  Finally,
  by Lemma~\ref{lem:upper-cover-of-type-III-IV}, we get
  $\tilde{m}=(k-1,t+1,t,\dotsc,t,\overset{c+1}{t-1},r)$
  for the upper cover of~$m$, which means that $\tilde{m}_1=k-1<k=g_1$ and
  thus $\tilde{m}\ngeq g$. Hence $g\Hochpfeil m$ becomes impossible by
  Proposition~\ref{prop:arrows-in-doubly-founded-lattices}\eqref{item:m-ucover-above-g}.
\end{proof}

Corollary~\ref{lem:single-typeC-I-II-described},
Lemma~\ref{lem:NecCondCone}
and Theorem~\ref{thm:TypeCdown1}
characterise all $m\in\Mi$ in relation
$g\Runterpfeil m$ with~$g$ of type~\ref{typeC}.
Dualising Theorem~\ref{thm:TypeCdown1} with the help of
Lemmata~\ref{Lemma_3.4} and~\ref{Lemma_3.5} yields the following result.
\begin{corollary}
  \label{corDualTypeCdown1}
  Let $k,d,t\in\Nplus$
  with $k\ge4$, $2\le t\le k-2$ and $2t\le d+1$ and set $n=k+d$.
  Moreover, decompose $n-(k-1) = ct +r$ with $c,r\in\N$ such that $0\le r<t$.
  Then $c\ge2$, and abbreviating $\kappa\defeq c+1$ we have
  $m^*=\big(\kappa+1,\dotsc,\overset{r}{\kappa+1},
            \kappa,\dotsc,\overset{t}{\kappa},
            1,\dotsc,\overset{k-1}{1}\big)
  \Hochpfeil
  g^* = \big(d+1,1,\dotsc,\overset{k}{1}\big)$,
  but not $m^* \DownArrow g^*$,
  where~$m^*$ is of type~\ref{typeD} and~$g^*$ is of type~\ref{typeIII}.
\end{corollary}

It remains to discuss
$g=(k+1,\dotsc,\overset{b}{k+1},k,\dotsc,\overset{b+\ell}{k},1,\dots,\overset{b+\ell+d}{1})\in\Ji$
of type~\ref{typeD}
where $b\in\N$, $k,\ell,d\in\Nplus$, $k\geq 3$, $b+\ell\geq 2$ and
$n=b(k+1)+\ell k+d$.
Then we may deduce via
Lemma~\ref{lem:single-typeC-b>=2-I-II} from $g\Runterpfeil m$ for some
$m\in\Mi$ that
$m=(a,t,\dotsc,\overset{c+1}{t},r)\in\Lp$ is of type~\ref{typeIV}
with $n=a+ct+r$, $a>t>r\geq0$, $c\geq 2$, $a=b+(b+\ell)(k-t)+t-1$ and $2\leq t\leq k-1$.
In this situation there are two cases: the first is that $t \leq d+1$,
in which Theorem~\ref{thm:double-typeCD} entails that
$g\Runterpfeil m$ is not only present, but that it is actually a double arrow.
The second case is that $d+1<t$, in which
Lemma~\ref{lem:single-typeC-b>=2-I-II} implies
$t>r=d+1\geq 2$, $b+\ell=c+1\geq 3$ and that $g\Hochpfeil m$ is impossible.
Under these necessary conditions, the following result shows that
$g\Runterpfeil m$ is actually present (and clearly is not a double
arrow).
\begin{theorem}\label{thm:TypeCDdown}
  Let $b\in\N$, $k,\ell,d,t\in \Nplus$ with
  $b+\ell\geq 3$ and $d+1<t\leq k-1$; set
  $n=b(k+1)+\ell k+d$ and $a=b+(b+\ell)(k-t)+t-1$, and consider
  $g=(k+1,\dotsc,\overset{b}{k+1},k,\dotsc,\overset{b+\ell}{k},1,\dotsc,\overset{b+\ell+d}{1})$
  and
  $m = (a,t,\dotsc,\overset{b+\ell}{t},d+1)$.
  Then we have $g\in\Ji$ of type~\ref{typeD} with $k\geq 4$,
  $m\in\Mi$ of type~\ref{typeIV} with $a\geq t+2$, and
  $g\DownArrow m$ but not $g\Hochpfeil m$.
\end{theorem}
\begin{proof}
  Since $d+2\leq t\leq k-1$ and $d\geq 1$, we have $t\geq 3$ and
  $k\geq d+3\geq 4$.
  Moreover, we observe
  \begin{multline*}
  a+(b+\ell-1)t+d+1 = b+(b+\ell)(k-t)+t-1 +(b+\ell)t-t+d+1
  =b+(b+\ell)k +d \\
  = b+bk +\ell k+d = b(k+1)+\ell k+d =n,
  \end{multline*}
  whence $m\in\Lp$.
  We check that
  $a=b+(b+\ell)(k-t)+t-1\geq  b+b+\ell+t-1\geq 0+3+t-1=t+2>t$
  as $k-t\geq 1$ and $b+\ell\geq3$.
  Hence, as $a>t>d+1\geq 2$ and $b+\ell\geq3$, we conclude that $m\in\Mi$ is
  of type~\ref{typeIV}.
  Since $k\geq 3$, $\ell,d\geq 1$ and $b+\ell\geq 2$, it is also clear
  that $g\in\Ji$ of type~\ref{typeD}.
  Furthermore, we have $g\nleq m$ because the partial sums up to
  position~$b+\ell$ satisfy
  $\hat{g}_{b+\ell}=n-d>n-(d+1)=\hat{m}_{b+\ell}$.
  Applying transition rule~\eqref{item:transition-cliff} to the cliff in
  position $b+\ell$ of~$g$, we obtain its unique lower cover
  $\tilde{g} =(k+1,\dotsc,\overset{b}{k+1},k,\dotsc,k,\overset{b+\ell}{k-1},2,1,\dotsc,\overset{b+\ell+d}{1})$
  as $d\geq 1$, with $1,\dotsc,1$ appearing for $d\geq2$.
  Note that the partial sums of both~$\tilde{g}$ and~$m$ reach
  $n-(d+1)$ at position~$b+\ell$.
  Therefore, as $\tilde{g}_i\ge k-1\ge t= m_i$ for $2\le i\leq b+\ell$,
  Lemma~\ref{lem:dominance-from-complement} with $j=b+\ell$ yields
  $\tilde{g}\leq m$, and hence
  Proposition~\ref{prop:arrows-in-doubly-founded-lattices}\eqref{item:g-lcover-below-m}
  implies $g\Runterpfeil m$.
  \par
  Finally, since $k\geq 3$, $b+\ell\geq 2$, $m\in\Mi$, $g\Runterpfeil m$ and $t>d+1$,
  Lemma~\ref{lem:single-typeC-b>=2-I-II} implies that $g\Hochpfeil m$
  is impossible.
\end{proof}

Lemma~\ref{lem:single-typeC-b>=2-I-II},
Corollary~\ref{cor:double-typeCD-described}
and Theorem~\ref{thm:TypeCDdown}
characterise all $m\in\Mi$ in relation $g\Runterpfeil m$ with~$g$ of
type~\ref{typeD}.
\par

The following summary combines Theorems~\ref{thm:TypeABDownIV}
and~\ref{thm:TypeCDdown}, covering~$g$ of types~\ref{typeA},
\ref{typeB} or~\ref{typeD}.
\begin{corollary}\label{cor:TypeABCDdownIVnew}
  Let $b,d\in\N$, $k,\ell\in \Nplus$ with
  $b+\ell\geq 3$ and $k\ge d+3$;
  set $n=b(k+1)+\ell k+d$.
  Furthermore, choose  $t\in\N$ with $d+1<t\leq k-1$
  and set $a\defeq b+(b+\ell)(k-t)+t-1$.
  Then we have
  $g=(k+1,\dotsc,\overset{b}{k+1},k,\dotsc,\overset{b+\ell}{k},1,\dotsc,\overset{\rlap{$\scriptstyle b+\ell+d$}}{1})\in\Ji$,
  ${m = (a,t,\dotsc,\overset{b+\ell}{t},d+1)\in\Mi}$
  of type~\ref{typeIV} with $a\geq t+2$,
  and~$g$ and~$m$ satisfy
  $g\DownArrow m$ but not $g\Hochpfeil m$.
\end{corollary}
Observe that $g=(4,4,4)$ satisfies
$g\DownArrow (5,3,3,1)$ and $g\DownArrow(7,2,2,1)$.
Hence Corollary~\ref{cor:TypeABCDdownIVnew}
represents a one-to-many relationship.
Similarly to Theorem~\ref{thm:TypeCdown1}
the number of distinct $m\in\Mi$ such that $g\DownArrow m$ with a
fixed~$g$ grows with~$n$, e.g., if $n=15$ and $g=(5,5,5)$, then
$g\DownArrow m$ for $m\in\set{(10,2,2,1), (8,3,3,1), (6,4,4,1)}$.
\begin{proof}
Note that $m\in\Mi$ is always  of type~\ref{typeIV}
while  $g\in\Ji$ is
  of type~\ref{typeA} (if $b=d=0$),
  of type~\ref{typeB} (if $b\ge1$, $d=0$),
  of type~\ref{typeD} (if $d\ge1$).
\par
For $d=0$ the hypotheses and partitions~$g$ and~$m$
simplify to those of Theorem~\ref{thm:TypeABDownIV},
since $d+1<t$ is equivalent to $2\le t$.
Therefore,
Theorem~\ref{thm:TypeABDownIV} proves $g\Runterpfeil m$ but not $g\Hochpfeil m$.
\par
For $d\ge1$ the hypotheses restate those of Theorem~\ref{thm:TypeCDdown}.
Hence $g\Runterpfeil m$ but not $g\Hochpfeil m$.
\end{proof}

Dualising Corollary~\ref{cor:TypeABCDdownIVnew} via
Lemmata~\ref{Lemma_3.4} and~\ref{Lemma_3.5} produces the following
result.
\begin{corollary}\label{corDualTypeABCDdownIV}
  Let $b,d\in\N$, $k,\ell,t\in \Nplus$ with
  $b+\ell\geq 3$, $k\ge d+3$ and $d+1<t\leq k-1$; further set
  $n=b(k+1)+\ell k+d$ and $a=b+(b+\ell)(k-t)+t-1$.
  Then $a\geq t+2$, and abbreviating $\kappa\defeq b+\ell$ we have
\[
 m^*=\big(\kappa+1,\dotsc,\overset{d+1}{\kappa+1}, \kappa,\dotsc,\overset{t}{\kappa}, 1,\dotsc,\overset{a}1\big)
 \Hochpfeil
 g^* = \big(\kappa+d,\,\kappa,\dotsc,\overset{k}{\kappa},\,b\big),
 \quad\text{but not}\quad m^* \Runterpfeil g^*,
\]
where $m^*\in\Ji$ is of type~\ref{typeD} and
$g^*\in\Mi$ is of type \ref{typeI} (if $b=d=0$),
\ref{typeII} (if $d=0,b\ge1$)
or~\ref{typeIV} (if $d\ge1$).
\end{corollary}

\begin{figure}[hp]
\centering
 \includegraphics[scale=1]{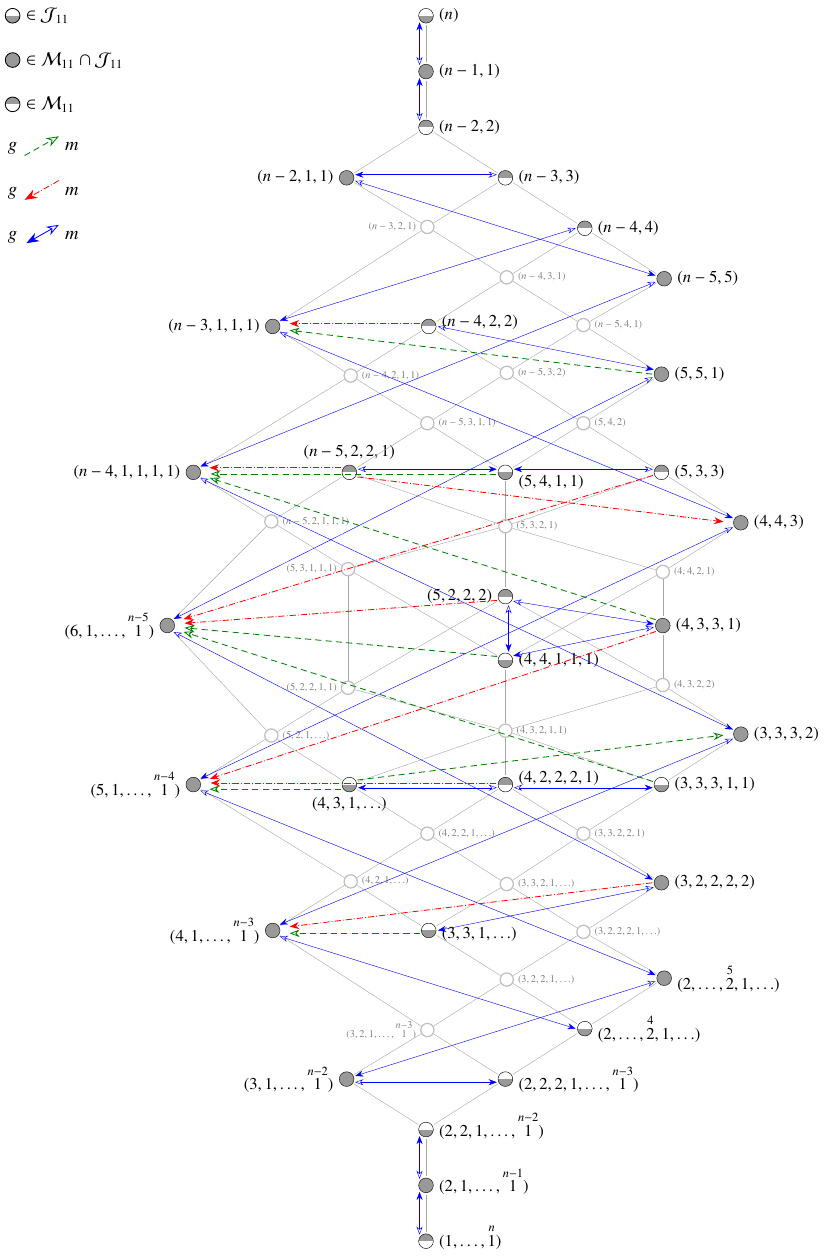}%
\caption{Lattice~$\Lp$ for $n=11$ with all arrows.}
\label{fig:lattice11}
\end{figure}

We finally collect our knowledge about single arrows in
\tablename~\ref{tab:OneDArrowCaracterizationSummary}.
Moreover, with \figurename~\ref{fig:lattice11} depicting~$\Lp[11]$ we aim to
illustrate for the concrete case $n=11$
the various up, down and double arrow relations appearing in
\tablename{s}~\ref{tab:DoubleArrowCaracterizationSummary}
and~\ref{tab:OneDArrowCaracterizationSummary}.
\begin{table}[ht!]
\centering
\begin{tabular}{llcl}
\toprule
Result & $g\in\Ji$ & Arrow & $m\in\Mi$ \\ \midrule
Theorem~\ref{thm:TypeCdown1}: &
 $(k,1,\dotsc,\overset{d+1}{1})$,\, $k\geq 4$, $d\geq 2t-1$
 & $\Runterpfeil, \nHochpfeil$ &
$(k-1, t,\dotsc,\overset{c+1}{t},r)$
\\ \midrule
Corollary~\ref{corDualTypeCdown1}:
 &
 $\big(\kappa+1,\dotsc,\overset{r}{\kappa+1},\kappa,\dotsc,\overset{t}{\kappa},1,\dotsc,\overset{k-1}{1}\big)$,\, $\kappa\defeq c+1$
& $\Hochpfeil,\nRunterpfeil$&
 $\big(d+1,1,\dotsc,\overset{k}{1}\big)$
\\ \midrule
Corollary~\ref{cor:TypeABCDdownIVnew}: &
$(k+1,\dotsc,\overset{b}{k+1},k,\dotsc,\overset{b+\ell}{k},1,\dotsc,\overset{b+\ell+d}{1})$,\,
$k\geq d+3$, $d\leq t-2$
 & $\Runterpfeil,\nHochpfeil$ &
$(a,t,\dotsc,\overset{b+\ell}{t},d+1)$
\\ \midrule
Corollary~\ref{corDualTypeABCDdownIV}: &
 $\big(\kappa+1,\dotsc,\overset{d+1}{\kappa+1},\kappa,\dotsc,\overset{t}{\kappa}, 1,\dotsc,\overset{a}{1}\big),$\, $\kappa\defeq b+\ell$
& $\Hochpfeil,\nRunterpfeil$ &
 $\big(\kappa+d,\,\kappa,\dotsc,\overset{k}{\kappa},\,b\big)$
\\
\bottomrule
\end{tabular}
\caption{A summary of the arrow characterisation results for~$\Hochpfeil$ and~$\Runterpfeil$
where $b,d\in\N$,  $c,\ell,t\in\Nplus$, $b+\ell\geq 3$ and $c,t\geq 2$.
There are different additional restrictions on~$t$ that have to be
taken from the respective results. Rows 1 and 2, and 3 and 4 are duals
of each other and are subject to the same conditions.}
\label{tab:OneDArrowCaracterizationSummary}
\end{table}

We are now ready to discuss once more the single arrows in
Example~\ref{ex:lattice7}.
The first $n\in\N$ where a single arrow appears in $\K(\Lp)$ is $n=7$.
This is justified by our Theorem~\ref{thm:TypeCdown1}
since its assumptions require $k\ge4$ and $4 \le 2t \le d+1$, which
gives $n = k+d \ge 4+3 = 7$.
Moreover, Corollary~\ref{cor:TypeABCDdownIVnew}
can only be applied for $n\ge9$ as it requires
$b+\ell\ge3$ and $k\ge d+3$, wherefore
$n=b + (b+\ell) k + d \ge  3k \ge 3(d+3) \ge 9$ as $b,d\ge0$.
The dual results Corollary~\ref{corDualTypeCdown1} and
Corollary~\ref{corDualTypeABCDdownIV}
are also valid for $n\ge7$ and $n\ge9$, respectively.
\par

Finally, for $n=7$, if $d=3$, then the restriction $4\le2t\le d+1=4$ of
Theorem~\ref{thm:TypeCdown1} shows that $t=2$, wherefore the partition
$g = (4,1,1,1)$ satisfies $g\DownArrow m$ with a unique $m\in\Mi[7]$.
Hence, considering the dual arrow (justified by
Corollary~\ref{corDualTypeCdown1})  there exist exactly two
single-directional arrows in $\K(\Lp[7])$. These two arrows are the
ones exhibited in Example~\ref{ex:lattice7} and illustrated in
\figurename~\ref{fig:lattice7}.

\section{One-generated arrow-closed one-by-one subcontexts}\label{sect:1-by-1-gen-ACSC}
We have implemented a brute-force algorithm (based on
Definition~\ref{def:arrow-rels}) and an algorithm that uses the
information given in
\tablename{s}~\ref{tab:DoubleArrowCaracterizationSummary}
and~\ref{tab:OneDArrowCaracterizationSummary}
to discover one-generated arrow-closed subcontexts of~$\K(\Lp)$.
Both implementations suggest that as of $n\ge3$ there exist exactly
$2n-4$ arrow-closed $(1\times1)$-subcontexts.
To demonstrate the applicability of our general characterisations of
the arrow relations in~$\K(\Lp)$ to the original problem of describing
one-generated arrow-closed subcontexts (with the ultimate intent to
obtain subdirect decompositions), we prove in this section that $2n-4$
is the correct number of one-generated arrow-closed subcontexts of
format~$(1\times 1)$.

Our first lemma deals precisely with all arrow-closed subcontexts
of~$\K(\Lp)$ generated by $g\in\Ji$ of types~\ref{typeA} or~\ref{typeB}.
\begin{lemma}\label{lem1x1AB}
 Let $b\in\N, \ell,k\in\Nplus$ be such that
 $b+\ell < n=b(k +1)+ \ell k$.
Then $g=(k+1,\dotsc,\overset{b}{k+1},k,\dotsc,\overset{b+\ell}{k})\in\Ji$
generates the following arrow-closed $(1\times1)$-subcontext
of\/~$\K(\Lp)$ with $m=(n-b-\ell,1,\dotsc,\overset{b+\ell}{1},1)\in\Mi$.
\[
\begin{array}{|c||c|}
  \hline
  & m=(n-b-\ell,1,\ldots,\overset{b+\ell}{1},1)\\
  \hline\hline
  g=(k+1,\ldots,\overset{b}{k+1},k,\ldots,\overset{b+\ell}{k}) & \Doppelpfeil\\
  \hline
\end{array}\]
\end{lemma}
\begin{proof}
Under the given conditions, $g$ is of type~\ref{typeA} or~\ref{typeB},
see Corollary~\ref{cor:least-bounded-width-typeAB}.
We apply the procedure to arrow-close the context.
The summarising
\tablename{s}~\ref{tab:DoubleArrowCaracterizationSummary}
and~\ref{tab:OneDArrowCaracterizationSummary} show that
$g\UpArrow m=(n-b-\ell,1,\dotsc,\overset{b+\ell}{1},1)$ by
Theorem~\ref{thm:double-typeAB} (in fact $g\Doppelpfeil m$)
and that this~$m$ is the only one in relation $g\Hochpfeil m$.
It remains to justify that no other partition $p\neq g$ satisfies
$p\DownArrow m$ for the given~$m$.
Considering the summarising
\tablename{s}~\ref{tab:DoubleArrowCaracterizationSummary}
and~\ref{tab:OneDArrowCaracterizationSummary}
and the shape of $m=(a,1,\dotsc,1)$, we realise that
Theorem~\ref{thm:double-typeAB} describes all $p\in\Ji$ such that
$p\DownArrow m$ (the cases $m=(1,\dotsc,1)$ and $m=(n-1,1)$ can be
handled by Theorem~\ref{thm:double-typeC}, as well, and yield the
same $p=g$ as Theorem~\ref{thm:double-typeAB}).
Therefore, the partitions satisfying $p\DownArrow m$
are  exclusively of type~\ref{typeA} or \ref{typeB} and have length $b+\ell= \len(m)-1$.
Then the uniqueness of the decomposition
$n = (b+\ell) k + b$ modulo $b+\ell$ where $0\le b<b+\ell$ entails
that~$p$ coincides with the given~$g$.
\end{proof}

We now dualise Lemma~\ref{lem1x1AB} and thereby cover in particular all
arrow-closed subcontexts of~$\K(\Lp)$ generated by any $g\in\Ji$ of
type~\ref{typeC}.
\begin{lemma}\label{lem1x1C1}
Let $d\in\N$, $k\in\Nplus$ with $k\geq 2$, set $n=k+d$ and
decompose $n= c(k-1) + r$ with $c,r\in\N$ such that $0\le r\le k-2$.
Then the partition $g=(k,1,\dotsc,\overset{d+1}{1})\in\Ji$
generates the following arrow-closed $(1\times1)$-subcontext
of\/~$\K(\Lp)$ with $m=(k-1,\dotsc,\overset{c}{k-1},r)\in\Mi$.
\[
\begin{array}{|c||c|}
  \hline
  &m=(k-1,\ldots,\overset{c}{k-1},r)\\
  \hline\hline
  g=(k,1,\ldots,\overset{d+1}{1}) & \Doppelpfeil\\
  \hline
\end{array}
\]
\end{lemma}
\begin{proof}
This result is the exact dual of Lemma~\ref{lem1x1AB}. An explicit
proof can be obtained by dualising the argument of Lemma~\ref{lem1x1AB}
where applications of Theorem~\ref{thm:double-typeAB} need to be replaced
by Theorem~\ref{thm:double-typeC} and types~\ref{typeA}
and~\ref{typeB} become types~\ref{typeI} and~\ref{typeII}.
The lemma can also be proved directly using the arrow-closing procedure
and the tables.
In fact, if $m\in\Mi$ is of type~\ref{typeI} or~\ref{typeII},
then \tablename{s}~\ref{tab:DoubleArrowCaracterizationSummary}
and~\ref{tab:OneDArrowCaracterizationSummary} confirm that
Theorem~\ref{thm:double-typeC} describes all $p\in\Ji$ such that
$p\Runterpfeil m$ (there are (only) two cases, $p=(n)$ and $p=(2,1,\dotsc,1)$,
which are handled by Theorem~\ref{thm:double-typeAB}
and~\ref{thm:double-typeC} simultaneously).
\end{proof}

\begin{proposition}\label{prop:NumberOf1x1acsc}
For $n\in\N$, $n\geq 3$ there are exactly $2n-4$ one-generated
arrow-closed $(1\times1)$-subcontexts in~$\K(\Lp)$.
\end{proposition}
\begin{proof}
Considering Lemmata~\ref{lem1x1AB} and~\ref{lem1x1C1}
we may count the arrow-closed $(1\times1)$-subcontexts.
First, Lemma~\ref{lem1x1C1} allows to choose the parameter
$k\in\set{2,\ldots,n}$. Hence, $n-1$ choices are available, and each leads
to a distinct~$g$ and thus a distinct $(1\times1)$-subcontext.
Moreover, Lemma~\ref{lem1x1AB} is exactly the dual of
Lemma~\ref{lem1x1C1}, wherefore we have once more~$n-1$ choices, and
thus $n-1$ distinct subcontexts, available.
The only $g\in\Ji$ that are captured by Lemma~\ref{lem1x1AB}
and~\ref{lem1x1C1} at the same time are either $g=(n)$, or have length
at least two and thus, according to Lemma~\ref{lem1x1C1}, end with~$1$,
have height at least~$2$ and satisfy $g_1>g_2$. Now, according to the
shapes of~$g$ considered in Lemma~\ref{lem1x1AB}, the height must be
equal to~$2$ since~$g$ ends in~$1$, and so $g=(2,1,1,\dotsc,1)$.
We therefore have to reduce the sum $2(n-1)$ by~$2$ and thus
have at least $2n-4$ distinct arrow-closed $(1\times1)$-subcontexts
generated by $g\in\Ji$ of types~\ref{typeA}, \ref{typeB} or~\ref{typeC}.
\par
Finally, if $g$ is of type~\ref{typeD} with $k\ge3$ and $b+\ell\ge2$, then
Theorem~\ref{thm:double-typeCD} exhibits some $m^{}_{\text{\ref{typeIV}}}\in\Mi$ of
type~\ref{typeIV} with $g\Hochpfeil m^{}_{\text{\ref{typeIV}}}$, and
Corollary~\ref{corDualTypeCdown1} yields some $m^{}_{\text{\ref{typeIII}}}\in\Mi$ of
type~\ref{typeIII} with $g\Hochpfeil m^{}_{\text{\ref{typeIII}}}$.
Since the type classes are disjoint, we therefore have two distinct
$m\in\Mi$ such that $g\UpArrow m$, and hence arrow-closed subcontexts
generated by~$g$ of type~\ref{typeD} cannot be of
format $1\times1$.
Consequently, the subcontexts shown in Lemmata~\ref{lem1x1AB}
and~\ref{lem1x1C1} are the only ones of this shape.
\end{proof}

\section{Conclusion}\label{sect:conclusion}
\begin{table}[b!]
  \centering
  \setlength{\tabcolsep}{12pt}
  \begin{tabular}{c@{\;}c@{\;}llr@{\;}c@{\;}ll}
  \toprule
  \multicolumn{3}{l}{Arrow vs.\ type}
  & \multicolumn{1}{l}{For}
  & \multicolumn{3}{c}{First example}
  & Result  \\
  \midrule
  \ref{typeA} & $\Doppelpfeil$ & \ref{typeI}   & $n=2$     &  $(2)$    & $\Doppelpfeil$ & $(1,1)$     & Theorem~\ref{thm:double-typeAB}       \\
  \ref{typeB} & $\Doppelpfeil$ & \ref{typeI}   & $n\ge3$     & $(2,1)$   & $\Doppelpfeil$ & $(1,1,1)$   & Theorem~\ref{thm:double-typeAB}       \\
  \ref{typeA} & $\Doppelpfeil$ & \ref{typeII}  & $n\ge3$     & $(3)$     & $\Doppelpfeil$ & $(2,1)$     & Theorem~\ref{thm:double-typeAB}       \\
  \ref{typeA} & $\Doppelpfeil$ & \ref{typeIII} & $n\ge4$     & $(2,2)$   & $\Doppelpfeil$ & $(2,1,1)$   & Theorem~\ref{thm:double-typeAB}       \\
  \ref{typeB} & $\Doppelpfeil$ & \ref{typeIII} & $n\ge5$     & $(2,2,1)$ & $\Doppelpfeil$ & $(2,1,1,1)$ & Theorem~\ref{thm:double-typeAB}       \\
  \midrule
  \ref{typeC} & $\Doppelpfeil$ &  \ref{typeI}   & $n\ge4$    & $(3,1)$   & $\Doppelpfeil$ & $(2,2)$     & Theorem~\ref{thm:double-typeC}    \\
  \ref{typeC} & $\Doppelpfeil$ &  \ref{typeII}  & $n\ge5$    & $(3,1,1)$ & $\Doppelpfeil$ & $(2,2,1)$   & Theorem~\ref{thm:double-typeC}    \\
  \midrule
  \ref{typeD} & $\Doppelpfeil$ &  \ref{typeIV}  & $n\ge7$    & $(3,3,1)  $ & $\Doppelpfeil$ & $(3,2,2)$    & Theorem~\ref{thm:double-typeCD}       \\
  \midrule
  \ref{typeC} & $\DownArrow$   &  \ref{typeIV}   & $n\ge7$   & $(4,1,1,1)$ & $\DownArrow$ & $(3,2,2)$    & Theorem~\ref{thm:TypeCdown1}          \\
  \ref{typeA} & $\DownArrow$   &  \ref{typeIV}   & $n\ge9$   & $(3,3,3)  $ & $\DownArrow$ & $(4,2,2,1)$     & Corollary~\ref{cor:TypeABCDdownIVnew}   \\
  \ref{typeB} & $\DownArrow$   &  \ref{typeIV}   & $n\ge10$  & $(4,3,3)  $ & $\DownArrow$ & $(5,2,2,1)$     & Corollary~\ref{cor:TypeABCDdownIVnew}   \\
  \ref{typeD} & $\DownArrow$   &  \ref{typeIV}   & $n\ge13$  & $(4,4,4,1)$ & $\DownArrow$ & $(5,3,3,2)$    & Corollary~\ref{cor:TypeABCDdownIVnew}   \\
  \midrule
  \ref{typeD} & $\UpArrow$     &  \ref{typeIII}  & $n\ge7$   & $(3,3,1)    $ & $\UpArrow$ & $ (4,1,1,1)$       & Corollary~\ref{corDualTypeCdown1}       \\
  \ref{typeD} & $\UpArrow$     &  \ref{typeI}    & $n\ge9$   & $(4,3,1,1)  $ & $\UpArrow$ & $ (3,3,3)$      & Corollary~\ref{corDualTypeABCDdownIV}   \\
  \ref{typeD} & $\UpArrow$     &  \ref{typeII}   & $n\ge10$  & $(4,3,1,1,1)$ & $\UpArrow$ & $ (3,3,3,1)$  & Corollary~\ref{corDualTypeABCDdownIV}   \\
  \ref{typeD} & $\UpArrow$     &  \ref{typeIV}   & $n\ge13$  & $(4,4,3,1,1)$ & $\UpArrow$ & $ (4,3,3,3)$  & Corollary~\ref{corDualTypeABCDdownIV}   \\
  \bottomrule
  \end{tabular}
  \caption{Arrow relation patterns between different types of
           partitions with least $n\in\N$ such that these occur
           in~$\K(\Lp)$.}
  \label{tab:arrow-type-patterns}
\end{table}

In this paper, for every $n\in\N$, we characterised all arrow relations
appearing in the standard context of the lattice~$\Lp$ of partitions
of~$n$ under the dominance order.
When looking at Tables~\ref{tab:DoubleArrowCaracterizationSummary}
and~\ref{tab:OneDArrowCaracterizationSummary}, collecting our results,
some curious patterns with respect to the three arrow shapes and the
partition types connected by arrows of a specific form arise. These
will be summarised subsequently; more
detailed information for which values of~$n$ these patterns appear and
which result justifies them, can seen from Table~\ref{tab:arrow-type-patterns}.
Inspecting Table~\ref{tab:arrow-type-patterns}, we note that all of
these connections can be observed in~$\K(\Lp)$ as soon as $n\geq 13$.
The following patterns caught our eye:
\begin{enumerate}[(a)]
\item If $g\in\Ji$ is of type~\ref{typeA}, \ref{typeB} or~\ref{typeC} and $m\in\Mi$ satisfies $g\Doppelpfeil m$, then~$m$ is
not of type~\ref{typeIV} and~$m$ is unique. Conversely, if $m\in\Mi$ is
of types~\ref{typeI}, \ref{typeII} or~\ref{typeIII} and
$g\Doppelpfeil m$ with $g\in\Ji$, then~$g$ is not of type~\ref{typeD}
and it is also unique. Therefore, the double arrows establish a
one-to-one relationship between elements of the union of the type
classes~\ref{typeA}--\ref{typeC} and the union
of the classes~\ref{typeI}--\ref{typeIII}.\par
More specifically, if $g\in\Ji$ is of type~\ref{typeA} or~\ref{typeB}
such that $g \neq (n)$, $g\neq (2,1,\dotsc,\overset{n-1}1)$
and $g\Doppelpfeil m$, then Theorem~\ref{thm:double-typeAB}
implies that~$m$ is of type~\ref{typeIII}, and the whole class of
type~\ref{typeIII} partitions is exhausted by those~$g$.
The two exceptional double arrow relations are
$(n)\Doppelpfeil (n-1,1)$
with~$g$ of type~\ref{typeA} and~$m$ of type~\ref{typeII}
(or~\ref{typeI} for $n=2$), and
$(2,1,\dotsc,1)\Doppelpfeil (1,\dotsc,1)$
with~$g$ of type~\ref{typeB} and~$m$ of type~\ref{typeI}.
In particular, $g\in\Ji$ of type~\ref{typeA} is never
$\Doppelpfeil$-connected to $m\in\Mi$ of type~\ref{typeI}
(or~\ref{typeIV}) when $n\geq 3$, and~$g$ of
type~\ref{typeB} is never $\Doppelpfeil$-connected to~$m$ of
type~\ref{typeII} (or~\ref{typeIV}).
In fact, $g$ of type~\ref{typeB} is never connected via any kind of arrow to~$m$ of type~\ref{typeII}, as shown in
Corollary~\ref{cor:BarrowIIimpossible}.
Having observed the bijective correspondence given by the double arrow
relation, this leaves that the $g\in\Ji$ of type~\ref{typeC} correspond in
a one-to-one way to the $m\in\Mi\setminus\set{(n-1,1),(1,\dotsc,1)}$ of
types~\ref{typeI} or~\ref{typeII}. This can also be verified directly
using Theorem~\ref{thm:double-typeC} for parameters $2\leq a\leq n-2$.
Note also that $g$ of type~\ref{typeC} is not even $\Hochpfeil$ or $\Runterpfeil$-related to~$m$ of type~\ref{typeIII}, as shown in Lemma~\ref{lem:CarrowIIIimpossible}.
\item If $g\in\Ji$, $m\in\Mi$ and $g\Doppelpfeil m$, then~$g$ is of
type~\ref{typeD} if and only if~$m$ is of type~\ref{typeIV}, see
Table~\ref{tab:DoubleArrowCaracterizationSummary}. However, this
correspondence is not unique because for sufficiently large values
of~$n$, $k$ and~$d$, Theorem~\ref{thm:double-typeCD} allows for
several choices of the parameter~$t$ determining distinct shapes
of $m\in\Mi$ of type~\ref{typeIV}. Dually, a given~$m$ of
type~\ref{typeIV} may only be double arrow related to $g\in\Ji$ of
type~\ref{typeD}, and these~$g$ may not be unique.
\item
Every up-arrow originating in~$g\in\Ji$ of type~\ref{typeA}, \ref{typeB}
or~\ref{typeC} is actually a double arrow described in
Table~\ref{tab:DoubleArrowCaracterizationSummary}. This is so because
Table~\ref{tab:OneDArrowCaracterizationSummary} does not contain any
up-arrows of such type.
Conversely, every down-arrow connecting $g\in\Ji$ with some $m\in\Mi$ of
type~\ref{typeI}, \ref{typeII} or~\ref{typeIII} is necessarily a double
arrow.
This implies that $g\in\Ji$ of types~\ref{typeA}, \ref{typeB}
or~\ref{typeC} receive down-arrows from $m\in\Mi$ that are not
double-arrows only if~$m$ is of type~\ref{typeIV}.
Dually, $g\in\Ji$ are in relation $g\Hochpfeil m$ and $g\nRunterpfeil m$
with $m\in\Mi$ of types~\ref{typeI}, \ref{typeII} or~\ref{typeIII} only
if~$g$ is of type~\ref{typeD}.
\par
On the other hand, $g\in\Ji$ of type~\ref{typeD} exhibit 
up-arrows to $m\in\Mi$ of all four types that fail to be double
arrows; likewise $m\in\Mi$ of type~\ref{typeIV} satisfy
$g\Runterpfeil m$ but $g\nHochpfeil$ with $g\in\Ji$ of all four
types. All of these single directional arrow relationships are in
general not one-to-one as Theorem~\ref{thm:TypeCdown1} and
Corollary~\ref{cor:TypeABCDdownIVnew}
exhibit a flexible parameter range for~$t$ with fixed and sufficiently
large~$k,d$ and~$n$. This happens as of~$n\geq 10$
for Theorem~\ref{thm:TypeCdown1} and as of~$n\geq 12$ for
Corollary~\ref{cor:TypeABCDdownIVnew}.
\end{enumerate}

As a side-product of Proposition~\ref{prop:NumberOf1x1acsc},
we discovered that the number
of ways a natural number~$n$ can be decomposed yielding partitions of
type~\ref{typeA} or~\ref{typeB} is exactly $n-1$, for this is the same
as the number of partitions of~$n$ of the shape
$(a,1,\dotsc,1)$ with $2\le a\le n$.
This is an example of a bijective proof as they are typical for
counting the number of restricted partitions in the combinatorial
theory of integer partitions, see, e.g.,
\cite[Section~2.2]{AndrewsErikssonIntegerPartitions}.
\par

Finally, in order to demonstrate the applicability of our general
descriptions, we started by characterising all one-generated
arrow-closed subcontexts of~$\K(\Lp)$ of format~$1\times 1$.
We determined that their number is exactly $2n-4$, see
Proposition~\ref{prop:NumberOf1x1acsc}. Working on this we observed
that also the characterisation of other arrow-closed subcontexts of
small shape seems to be within reach of our results.
We leave the complete description of the one-generated arrow-closed
subcontexts of~$\K(\Lp)$ as a task for future investigation.


\end{document}